\documentclass[letterpaper, 11pt,  reqno]{amsart}
\usepackage{amsmath,amssymb,amscd,amsthm,amsxtra, esint}
\usepackage{tikz} 
\usepackage{comment}
\usepackage{mathtools}
\usepackage{color}
\usepackage[implicit=true]{hyperref}
\usepackage{scalerel} 
\usepackage{cases}
\usepackage{enumerate}
\usepackage{soul}
\usepackage[margin=1.2in,marginparwidth=1.5cm, marginparsep=0.5cm]{geometry}
\allowdisplaybreaks[2]
\usepackage{color}
\definecolor{gr}{rgb}   {0.,   0.69,   0.23 }
\definecolor{bl}{rgb}   {0.,   0.5,   1. }
\definecolor{mg}{rgb}   {0.85,  0.,    0.85}
\definecolor{yl}{rgb}   {0.8,  0.7,   0.}
\definecolor{or}{rgb}  {0.7,0.2,0.2}

\usetikzlibrary{shapes.misc}
\usetikzlibrary{shapes.symbols}
\usetikzlibrary{shapes.geometric}

\tikzset{
	ddot/.style={circle,fill=white,draw=black,inner sep=0pt,minimum size=0.8mm},
	>=stealth,
	}

\tikzset{
	ddot2/.style={circle,fill=black,draw=black,inner sep=0pt,minimum size=0.8mm},
	>=stealth,
	}
\newtheorem{theorem}{Theorem} [section]

\newtheorem{lemma}[theorem]{Lemma}
\newtheorem{proposition}[theorem]{Proposition}
\newtheorem{remark}[theorem]{Remark}

\newcommand{\1}{\hspace{0.5mm}\text{I}\hspace{0.5mm}}
\newcommand{\II}{\text{I \hspace{-2.8mm} I} }
\newcommand{\III}{\text{I \hspace{-2.9mm} I \hspace{-2.9mm} I}}
\newcommand{\noi}{\noindent}
\newcommand{\Z}{\mathbb{Z}}
\newcommand{\R}{\mathbb{R}}

\newcommand{\T}{\mathbb{T}}

\let\Re=\undefined\DeclareMathOperator*{\Re}{Re}
\let\Im=\undefined\DeclareMathOperator*{\Im}{Im}
\let\P= \undefined
\let\om=\omega
\newcommand{\P}{\mathbf{P}}

\newcommand{\E}{\mathbb{E}}

\newcommand{\al}{\alpha}

\newcommand{\dl}{\delta}

\newcommand{\eps}{\varepsilon}

\newcommand{\g}{\gamma}

\newcommand{\Si}{\Sigma}
\newcommand{\ft}{\widehat}

\newcommand{\wt}{\widetilde}
\newcommand{\cj}{\overline}

\newcommand{\dt}{\partial_t}

\renewcommand{\o}{\omega}

\def\cI{{\mathcal I}}

\renewcommand{\P}{\mathbb P}

\newcommand{\les}{\lesssim}

\newcommand{\ind}{\mathbf 1}

\newcommand{\pa}{\partial}

\def\e{\eps}

\renewcommand{\H}{\mathcal{H}}

\newcommand{\HS}{\textup{HS}}
\newtheorem*{ackno}{Acknowledgements}

\numberwithin{equation}{section}
\numberwithin{theorem}{section}

\makeatletter

\begin{document}
\baselineskip = 13.5pt

\title[GWP of the energy-critical SHNLW]
{Global well-posedness of the energy-critical 
stochastic Hartree nonlinear wave equation} 

\author[G.~Li, L.~Tao, and T.~Zhao]
{Guopeng Li, Liying Tao, and Tengfei Zhao}

\address{
Guopeng Li, 
School of Mathematics and Statistics, Beijing Institute of Technology,
Beijing 100081,
China and School of Mathematics,
The University of Edinburgh,
and The Maxwell Institute for the Mathematical Sciences,
James Clerk Maxwell Building,
The Kings Buildings,
Peter Guthrie Tait Road,
Edinburgh,
EH9 3FD,
United Kingdom
}
\email{guopeng.li@bit.edu.cn}

\address{Liying Tao,
Graduate School of China Academy of Engineering Physics, Beijing, 100088, China}
\email{taoliying20@gscaep.ac.cn}

\address{Tengfei Zhao,
School of Mathematics and Physics, University of Science and Technology Beijing, Beijing, 100083, China}
\email{zhao\underline{ }tengfei@ustb.edu.cn}%



\begin{abstract}
We consider the Cauchy problem for the stochastic Hartree nonlinear wave equations~(SHNLW)
with a cubic convolution nonlinearity
and an additive stochastic forcing
on the Euclidean space. 
Our goal in this paper is two-fold.
(i)~We study the defocusing
energy-critical SHNLW on $\mathbb R^d$, for $d \geq 5$,
and prove that they are globally well-posed  
with deterministic initial data in the energy space. 
(ii)~Next, we consider the well-posedness of the defocusing energy-critical SHNLW with randomized initial data
below the energy space.
In particular, when $d=5$, we prove that it is almost surely globally well-posed.  
As a byproduct,
by removing the stochastic forcing
our result covers 
the study of the (deterministic) 
Hartree nonlinear wave equation~(HNLW)
with randomized initial data below the energy space.

The main ingredients
in the globalization argument
involve 
the probabilistic perturbation approach by  
B\'enyi-Oh-Pocovnicu~(2015) and Pocovnicu~(2017), time integration by parts trick
of Oh-Pocovnicu~(2016),
and an estimate of the Hartree potential energy.

%
 
\end{abstract}

\maketitle

\tableofcontents

\section{Introduction}

\subsection{Stochastic Hartree nonlinear wave equation}
We study the following Cauchy problem 
for the defocusing energy-critical 
nonlinear wave equation (NLW) with a stochastic forcing
on $\R^d$ for $d\geq 5$:
\begin{equation}
\label{eq_SHW}
\begin{cases}
\pa_t^2 u-\Delta u+
(|\cdot|^{-4}\ast u^2)u=\phi\xi   \\
(u,\pa_tu)|_{t=0}=(u_0,u_1),
\end{cases}
\quad (t,x)\in\R_+\times\R^d,
\end{equation}

\noi
where the unknown
$u$ is a real-valued function,
$\xi$ is the space-time white noise on $\R_+\times\R^d$, and
$\phi$ is a bounded operator on $L^2(\R^d)$.
The nonlinearity we consider in \eqref{eq_SHW} is of the Hartree type,
namely, it is given by a cubic convolution, 
see \eqref{V} below. 
For this reason, 
we refer to \eqref{eq_SHW} as the
stochastic {\it Hartree} nonlinear wave equation (SHNLW).
We say that $u$ is 
a solution to \eqref{eq_SHW}, if it satisfies the
following Duhamel formulation (= mild formulation):
\begin{align}
\label{Duhamel}
\begin{aligned}
u(t)=&S(t)(u_0,u_1)
-\int_0^t 
Q(t-t') 
(|\cdot|^{-4}\ast u^2)u(t')dt' +
\int_0^t 
Q(t-t') 
\phi\xi(dt'),
\end{aligned}
\end{align}

\noi
where $Q(t)$  and $S(t)$ are defined by
\begin{equation}
\label{S(t)}
Q(t)= 
\frac{\sin( t |\nabla|)}{|\nabla|} 
\quad
\text{and}
\quad
S(t)(\phi_0, \phi_1)
=  \dt Q(t) \phi_0
+Q(t)  \phi_1.
\end{equation}

\noi
The last term on the right-hand side of 
\eqref{Duhamel}
represents the effect of stochastic forcing and 
is known as
the stochastic convolution:
\begin{equation}
\label{Psi}
\begin{split}
\Psi(t):= \int_0^t 
Q(t-t')  
\phi\xi(dt'),
\end{split}
\end{equation}
which solves the linear stochastic wave equation:
\begin{equation*}
\begin{cases}
\pa_t^2\Psi-\Delta \Psi=\phi\xi\\
(\Psi,\dt\Psi)|_{t=0}=(0,0).
\end{cases}
\end{equation*}

\noi For a precise definition of the stochastic convolution $\Psi$, see Subsection \ref {Subs-sto}.
In order to state our result, we record the following:
If $\phi\in \HS(L^2;H^s)$ for some real number $s$, it 
is known that $\Psi\in C(\R_+;H^{s+1}(\R^d))$ almost surely; see \cite{DZ14, BLL23}.

Our main goal in this paper is to prove  almost sure global well-posedness of~\eqref{eq_SHW}
in the following two situations:
\begin{itemize}
\item[(i)]
with deterministic initial data in the energy space
and $\phi  \in \HS(L^2;L^2)$, see Theorem~\ref{thm-d-i};

\smallskip
\smallskip
\item[(ii)]
with randomized initial data below the energy space.
In this case, however, we need to take a smoother noise, see Theorem~\ref{thm-r-i}.

\end{itemize}

Let us first go over some backgrounds on the Hartree NLW on $\R^d$:
\begin{equation}
\label{Hartree-Wave}
\pa_t^2 u-\Delta u+(V\ast u^2)u=0,
\end{equation}
where $V$ is given by
\begin{equation}\label{V}
V(x)=\frac{\mu}{|x|^{\gamma}},
\qquad
\mu\in\R,\quad x\in\R^d,\quad  0<\gamma<d.
\end{equation}

\noi
Here, the $V$ describes the two-body interaction potential between atoms 
%
and is known as the {\it Hartree potential}, 
which has garnered significant interest in the physics community; see, for example, \cite{H28, G60}.

In this paper, we focus on the 
{\it energy-critical}
setting by taking $\mu=1$ and $\gamma=4$ in \eqref{V}.
%
In this case, the equation \eqref{Hartree-Wave} reduces to the following Hartree nonlinear wave equation (HNLW):

\noi
\begin{equation}
 \label{HNLW}
  \pa_t^2 u-\Delta u+( |\cdot|^{-4} \ast u^2)u=0,
\end{equation}

\noi
with 
the associated energy (= Hamiltonian) given by 
\begin{equation}\label{energy}
E(u,\pa_t u)=
\frac{1}{2}\int_{\R^d}(|\pa_t u(x)|^2+|\nabla u(x)|^2)dx+\frac{1}{4}\iint_{\R^d\times \R^d}\frac{|u(x)|^2|u(y)|^2}{|x-y|^4}dxdy.
\end{equation}

\noi
The equation 
\eqref{HNLW} 
enjoys the following scaling symmetry 
\begin{equation*}
u(t,x)\mapsto \lambda^{\frac{d-2}{2}}u(\lambda t,\lambda x)
\quad 
\text{for}
\quad
\lambda>0,
\end{equation*}

\noi
which also preserves the functional $E(u,\dt u)$.
For this reason,  
the equation \eqref{HNLW} is known as energy-critical
with energy space $\dot\H^1(\R^d)=\dot H^1(\R^d)\times L^2(\R^d)$. 
While SHNLW \eqref{eq_SHW} does not enjoy the scaling symmetry (when $\phi \neq 0$), 
we still refer to it as energy-critical.
The  heuristic provided by the scaling symmetry
dictates that  
\eqref{HNLW} is well-posed in 
$\dot \H^s$
for $s\ge 1$  and ill-posed in $\dot\H^s$ for $s<1$. 



Before discussing 
the well-posedess results of
HNLW \eqref{HNLW}, 
let us review some deterministic well-posedness theory of NLW with power-type nonlinearity.  
%
This  area has been studied extensively since the 1960s
and it is known that 
the defocusing energy-critical NLW is global well-posed in the energy space and the solutions scatters;
 see~ 
\cite{S68, S88, G90, K94, SS94, BG99, Tao06}.
On the other hand, when $s<1$, it is known to be ill-posed in $\H^s(\R^d):= H^s(\R^d)\times H^{s-1} (\R^d)$; 
see \cite{BK00, CCT03, Lebeau05, IM07, FO20, OOTz} for related ill-posedness results example.
%
%

%
In 1982,  Menzala and Strauss \cite{MS82} studied 
the well-posedness and the scattering theory  for HNLW; 
see also \cite{M89, H00, CG13, MZZ15}. 
More recently, Miao, Zhang, and Zheng 
\cite{MZZ14} 
proved the global well-posedness and 
scattering theory for the defocusing energy-critical HNLW \eqref{HNLW} in the energy space.
%
%
%
%

%
%

A primary motivation for our work is to investigate the defocusing energy-critical Hartree-type nonlinear wave equation (HNLW) in a probabilistic setting, incorporating both stochastic forcing and randomized initial data. For related results on energy-critical stochastic nonlinear wave equations and stochastic nonlinear Schrödinger equations, we refer the reader to \cite{OO20, CL22, BLL23, Zhang23}.

In this manuscript, we focus primarily on the nonlocal potential in the energy-critical case, namely, the nonlinearity of the form
$(|\cdot|^{-4} \ast u^2)\, u$.
We note that from the definition of the  nonlocal potential \eqref{V},
we  must have $0<\g<d$, for $d=$ dimensions.
Hence, when $\g=4$ our dimensions must be strictly greater than $4$,
which imposes the dimension restriction 
such that  $d \geq 5$.
Namely, 
 this (nonlocal) Hartree-type nonlinearity is well-defined only 
 for dimensions $d \geq 5$; in lower dimensions $3 \leq d \leq 4$, 
 the convolution becomes ill-defined due to singularity issues.

\subsection{Main results}
 
In recent years, there has been significant progress in the study of well-posedness for the stochastic nonlinear wave equation (SNLW) with rough stochastic forcing; see
\cite{OOR20, OO20, CL22, GKOT22, OPT22, OWZ22, BLL23, GKO23, ORT23}, and the references therein. See also the survey \cite[Chapter 13]{DZ14}, for example.
In particular, Brun, Liu, and the first author \cite{BLL23} studied the defocusing energy-critical SNLW and proved its global well-posedness.

On the other hand, in his seminal works \cite{B94, B97}, Bourgain initiated the study of dispersive PDEs with random initial data.
Over the past two decades, this direction has seen substantial developments
\cite{Tz06, BT07, BT08-2, Oh09, CO12, D12, NORS12, BT14, LM14, LM16, SX16, CCMNS20, Bringmann21, Latocca21, KOC22}.
Regarding the construction of local-in-time solutions, randomization of the initial data allows one to go beyond the limitations of deterministic analysis, enabling the construction of solutions even in regimes where the equation is ill-posed; see, for example, Burq-Tzvetkov \cite{BT08-1}.
In particular, for the closely related works,
 almost sure global well-posedness of the energy-critical NLW below the energy space 
 (with random initial data), we refer to the pioneering works of Pocovnicu \cite{P17} and Oh-Pocovnicu \cite{OP16}.

Our main results are twofold:
\begin{itemize}
\item[(1)] Almost sure global well-posedness of \eqref{eq_SHW} with deterministic initial data in the energy space (Theorem \ref{thm-d-i}), which extends the result of \cite{BLL23} (for power-type nonlinearities) to the nonlocal Hartree-type nonlinearity.

\item[(2)] Almost sure global well-posedness of \eqref{eq_SHW} below the energy space (on $\R^5$), which extends the results of \cite{P17, OP16} to the setting of a nonlocal (Hartree) nonlinearity and forced by additive noise, and it is unknown in the study of  \cite{BLL23}.

\end{itemize}



\medskip
\noi
{\bf Energy-critical SHNW with deterministic initial data on $\R^d$ for $d\geq 5$.}

Our first goal is to  extend  \cite{BLL23}  to the nonlocal (Hartree) nonlinearity.


%

\begin{theorem}
\label{thm-d-i}
Let $d\geq 5$ and $\phi\in \HS(L^2(\R^d);L^2(\R^d))$. 
Then, the defocusing energy-critical  stochastic Hartree nonlinear wave equation \eqref{eq_SHW}
is  globally well-posed in $\dot\H^1(\R^d)$. 



\end{theorem}

As mentioned earlier, Theorem \ref{thm-d-i} complements to existing global well-posedness results of  energy-critical stochastic nonlinear dispersive PDEs. 
In particular, this is the first result for the nonlocal nonlinearity, to our knowledge.

When $\phi=0$, 
SHNLW \eqref{eq_SHW} reduces to the HNLW \eqref{HNLW}. 
Therefore, 
similar to the discussion of the ill-posedness of equation \eqref{HNLW} 
 in $\dot \H^s(\R^d)$ for $s<1$,
 we anticipate that 
Theorem \ref{thm-d-i} is sharp in terms of the regularity of the initial data. 
%
%
On the other hand, the restriction on the regularity of $\phi$ comes from  the
use of the Strichartz estimate
in studying the space-time 
regularity of the stochastic convolution; see Lemma \ref{regularity of Psi}.

Let us briefly discuss the idea  of proof of Theorem \ref{thm-d-i}.
By writing the solution $u$ to \eqref{eq_SHW}
in first order expansion 
\cite{Mckean95, B96, DD02}: 
\begin{align*}
u=v+\Psi,
\end{align*}
we are reduced to showing the global well-posedness of the following
perturbed HNLW satisfied by the residual term $v$:
\begin{equation}
\label{perturbed v}
\begin{cases}
\pa_t^2 v-\Delta v+\mathcal{N}(v+\Psi)=0\\
(v,\pa_t v)|_{t=0}=(u_0,u_1)
\end{cases}
\end{equation}

\noi
in $\dot \H^1(\R^d)$, 
where $\Psi$ is the stochastic convolution given by  \eqref{Psi} and $\mathcal{N}(u)=(|\cdot|^{-4}\ast|u|^2)u$.
Hence, the uniqueness in Theorem \ref{thm-d-i} refers to the part $v$ of a solution $u$.

%





Combining the Strichartz estimates and the
space-time regularity of  stochastic
convolution (Lemma~\ref{regularity of Psi}),
one can easily prove  \eqref{perturbed v} is local well-posed, that is, for almost all 
$\omega$, there exists  
$T>0$ such that there exists a unique 
$\vec v$ in a subspace of $C([0,T); \dot \H^1(\R^d))$, 
where $\vec v=(v,\dt v)$.
The main difficulty of Theorem~\ref{thm-d-i} lies in establishing the global well-posedness of \eqref{eq_SHW},
which is reduced to global well-posedness of \eqref{perturbed v} for the residual term $v$. 
Then, by employing Ito's lemma 
and the fact that
 $\phi\in \HS(L^2(\R^d);L^2(\R^d))$, 
we can show an a priori bound of the solution $v$ to \eqref{perturbed v}
\begin{align}
\label{v-ebd}
\| (v(t), \dt v(t)) \|_{L^{\infty}([0,T]; 
\dot \H^1(\R^d)) } 
\leq C(\omega,T,\|\phi\|_{\HS(L^2(\R^d);L^2(\R^d))})<\infty 
\end{align}
for almost all $\omega\in \Omega$ and any $T>0$.

Due to the energy-critical nature of the equation,
having an a priori bound on the $\dot\H^1(\R^d)$ norm of $\vec v$ is not sufficient.
Instead, we need to have a control on the growth of the Strichartz norm of $v$, 
for which we adapt the   
{\it probabilistic perturbation theory}
introduced in 
B\'enyi, Oh, and Pocovnicu \cite{BOP15'}
 and   Pocovnicu \cite{P17}; 
see Lemma \ref{perturbation} below. 
%
%
%
%
We then implement an iterative argument,
combining with the perturbation lemma 
 and the global space-time bound 
(Lemma \ref{st-bound}),  
and prove a global-in-time existence of a solution 
$u$ in the following class
\begin{equation*}
(\Psi,\pa_t\Psi)+C(\R_+;\dot\H^1(\R^d))\subset C(\R_+;\dot\H^1(\R^d)).
\end{equation*}




%

\begin{remark}\rm

Compared to \cite{BLL23}, there are two main differences:
(1) First, we are able to handle a more complicated nonlinearity in the current manuscript.
(2) Second, on $\mathbb{R}^d$, \cite{BLL23} does not establish almost sure global well-posedness for the SNLW with randomized initial data below the energy space. In contrast, our second main result (Theorem \ref{thm-r-i}) proves almost sure global well-posedness for \eqref{eq_SHW} on $\mathbb{R}^5$ with Wiener-randomized initial data below the energy space; see also \cite[Remark 1.3]{BLL23}.
\end{remark}

\begin{remark}\label{remark-Td}
\rm

The reason we can handle rough stochastic noises
in Theorem \ref{thm-d-i} 
is due to the integrability gain of the wave propagator;
see Lemma \ref{regularity of Psi}.
Moreover, as considered in ~\cite{BLL23},
for the case of torus $\T^d$, we may expect
$\phi$ rougher due to the boundedness of the torus, 
which gives better spatial integrability of the 
Stochastic convolution; see also \cite{OOPT23}.


\end{remark}

\begin{remark}
\rm
The perturbation theory has played an important role in the study of deterministic energy-critical NLW and 
nonlinear Schr\"odinger euqation (NLS); see, for example,~\cite{TVZ07, CKSTT08, KM08,  MXZ11, KOPV12, MZZ14}.
This perturbation theory has also been successfully adapted to the probabilistic setting in various works
\cite{BOP15, 
LM16, OP16, 
 OP17, P17,
 BOP-19, BOP-19-2, OOP19,
 DLM20}.


\end{remark}


%
%

\medskip
\noi
{\bf SHNLW with Wiener randomized initial data {\it below} energy space on $\R^5$.}

Our next goal is to study the Cauchy problem \eqref{eq_SHW}
with randomized initial data  below the energy space,
which is excepted to be {\it ill-posed} 
for the deterministic initial data; see  discussion in Remark \eqref{RM:ill}.

Before we state our second main result,
let us now briefly discuss our randomized initial data.
More precisely,
given a pair of functions $(u_0,u_1)$ on $\R^d$, 
we define the following 
Wiener decomposition (see \cite{LM14, BOP15, BOP15'}) by setting
\begin{equation}
\label{R (u0,u1)}
(u_0^\omega, u_1^\omega)
=\Big(\sum_{n\in\Z^d} g_{n,0}(\omega)\psi(D-n)u_0,
\sum_{n\in\Z^d}g_{n,1}(\omega)\psi(D-n)u_1\Big), 
\end{equation}

\noindent
where $\{g_{n,j}\}_{n\in \Z^d}$, $j=1,2$ 
are mean-zero random variables and  independent of the noise
$\xi$ 
with suitable conditions;  see Subsection \ref{SUB:WR} for further details.
Here, $\psi(D-n)$ is the smooth Fourier multiplier onto the cube $[-1,1]^d+n$ for $n\in \Z^d.$ 
%
Given $(u_0,u_1) \in \H^s(\R^d) $ for some appropriate $s\geq0$,
it is easy to see that its Wiener randomization 
\eqref{R (u0,u1)} lies in $\H^s(\R^d)$ almost surely.
One can also show that
there is no smoothing upon randomization in terms of differentiability.
However, the main advantage of the randomization in  \eqref{R (u0,u1)} is the gain of the integrability;
see Remark \ref{ramdf}.
This gain of space-time integrability allows us to take random
initial data below the energy space.

\begin{theorem}
\label{thm-r-i}
\textup{(i)}
 ~\textup{(Local well-posedness)}
Let   $0\leq s<1$ and  $\phi\in \HS(L^2(\R^5); L^2(\R^5))$.
Given~$(u_0,u_1)\in \dot \H^s(\R^5)$, let $(u_0^\o,u_1^\o)$ be its Wiener randomization 
  defined in \eqref{R (u0,u1)} satisfying \eqref{dis}.
Then the stochastic Hartree nonlinear wave equation \eqref{eq_SHW} is almost surely locally well-posed 
with respect to the randomization 
$(u_0^\omega, u_1^\omega)$
as initial data.
More precisely, 
there exists a (unique) solution $u$ to  \eqref{eq_SHW} with $(u,\dt u)|_{t=0}= (u_0^\o,u_1^\o)$ in the class:
\begin{equation*}
(S(t) (u_0^\o,u_1^\o),\dt S(t) (u_0^\o,u_1^\o))+(\Psi,\pa_t\Psi)+C([0,T];\dot\H^1(\R^5))
\subset C([0,T]; \dot \H^s(\R^5)),
\end{equation*}
where $T=T_\o$ is almost surely positive.
%
%

\vspace{0.4cm}
\textup{(ii)}~\textup{(Global well-posedness)}
Let $\frac12<s<1$ and  $\phi\in \HS(L^2(\R^5); H^1(\R^5))$.
Given~$(u_0,u_1)\in \H^s(\R^5)$, let $(u_0^\o,u_1^\o)$ be its Wiener randomization   defined 
in \eqref{R (u0,u1)} satisfying \eqref{dis}.
Then,
 the defocusing energy-critical stochastic Hartree nonlinear wave equation~\eqref{eq_SHW}
  is almost surely globally well-posed
with respect to the randomization  $(u_0^\omega, u_1^\omega)$ 
as initial data.
More precisely, 
there exists a  set
 $\Omega_{(u_0,u_1)} \subset \Omega$  
 of  probability $1$ such that, 
for any $\omega \in \Omega_{(u_0,u_1)}$,
there exists
a (unique) solution 
 $u$ to~\eqref{eq_SHW}
in the class:
\noindent
\begin{equation*}
(S(t) (u_0^\o,u_1^\o),\dt S(t) (u_0^\o,u_1^\o))
+(\Psi,\pa_t\Psi)+C(\R_+;\H^1(\R^5)) \subset C(\R_+; \H^s(\R^5)).
\end{equation*}

%

\vskip 0.3cm

 \end{theorem}




In Theorem \ref{thm-r-i}, we allow randomized initial data below the energy space.
Our result considers the SHNLW \eqref{eq_SHW} with an additional noise term, thereby extending the work of \cite{OP16} (which has no forcing term).
Moreover, the additional regularity assumption on $\phi$ in Theorem~\ref{thm-r-i}~(ii) is required to establish global well-posedness.


The main strategy for proving Theorem \ref{thm-r-i} is based on 
the following first order extension 
 $$u=v+\wt \Psi,$$

\noi
where now $\wt \Psi$ solves 
\begin{align}
\label{Psi2}
\begin{cases}
\pa_t^2 \wt \Psi-\Delta \wt \Psi= \phi\xi\\
(\wt \Psi, \dt \wt \Psi)|_{t=0}=
(u_0^\o,u_1^\o).
\end{cases}
\end{align}

\noi
Namely, we have 
$$\wt \Psi= S(t) (u_0^\o,u_1^\o) +\Psi.$$

\noi 
Regarding Theorem \ref{thm-r-i}~(i), the main observation is that
$\wt \Psi$ satisfies the same space-time 
integrability of $\Psi$; see Lemma \ref{regularity of Psi} and Lemma \ref{P-est-S*}. 
Therefore, the  argument for local well-poseness 
of~ \eqref{eq_SHW} is identical to  that of 
Theorem \ref{thm-d-i} by studying 
the following equation of the residue term $v$ in $C([0,T];\dot\H^1(\R^5))$
\begin{equation}
\label{eq_SHW3}
\begin{cases}
\pa_t^2 v -\Delta v +
\mathcal{N}(v+\wt \Psi)
=0\\
(v,\pa_t v)|_{t=0}=(0,0).
\end{cases}
\end{equation}
See Remark \ref{RM:local}.

For the global well-posedness part, the main strategy 
is the same as that for proving Theorem~\ref{thm-d-i}. Namely, 
by having an a priori
energy bound \eqref{v-ebd},
we can prove Theorem~\ref{thm-r-i}~(ii) by employing the 
probabilistic perturbation argument.
The main difficulty appears in establishing the a priori energy bound~\eqref{v-ebd} on the growth of the solution $v$ to \eqref{eq_SHW3},
where we now need to handle the stochastic convolution and randomized linear solution simultaneously.
In the current setting, since the initial data does not belongs to energy space, we need to adapt
the probability tools for probability well-posedness theory of NLW. 
In \cite{BT14},
Burq and Tzvetkov introduced a 
Gronwall-type argument to obtain the energy bound of the $v$ in the context of defocusing
cubic NLW on $3$d boundaryless Riemannian manifold;
see also \cite{P17, OP17}.
However, in current setting, for the Hartree
nonlinearity, as in \cite{OP16},
the Gronwall-type argument does not seem to work.
In order to overcome this difficulty, we combine 
the time integration by parts trick introduced by \cite{OP16} (see Proposition~\ref{D-bound}) together with following 
estimate of the Hartree potential energy (see Lemma~\ref{Riesz-Cha}): for $d\geq 5$
\[
\Big\||\nabla|^{-\frac{d-4}{4}} u\Big\|^4_{L_{x}^4(\R^d)}
\les
\iint_{\R^d\times\R^d}\frac{|u(x)|^2|u(y)|^2}{|x-y|^4} dx dy.
\]

\noindent
We note that it is in this step that we need 
the extra regularity of $\phi$,  which ensures the space-time integrability  of  $\Psi$ and $\dt \Psi$; see Lemma~\ref{regularity of Psi}~(iv).

\medskip
We conclude this introduction by stating the following remarks.

\begin{remark} \rm
\label{RM:ill}

Note that we previously mentioned the defocusing energy-critical NLW with a power-type nonlinearity is ill-posed below the energy space.
In our setting with a Hartree-type nonlinearity, 
we similarly expect the defocusing energy-critical HNLW \eqref{HNLW} to be ill-posed in $\H^s(\R^d)$ for $s < 1$.
However, the existing methods {\it do not} seem to be directly applicable to \eqref{HNLW} due to the complexity of the nonlocal nonlinearity.
It would be interesting to investigate the ill-posedness of the defocusing energy-critical HNLW \eqref{HNLW} below the energy space; see \cite{MXZ08} for ill-posedness results in the Schrödinger case.

\end{remark}

\begin{remark}
\rm

(i) Theorem \ref{thm-r-i}, extends the results in \cite{P17, OP16} (for power-type nonlinearities and {\it no noise term}) to the setting with a nonlocal nonlinearity and additive noise.
We note that Theorem \ref{thm-r-i} (ii) is restricted to the case $d = 5$, due to the use of Gronwall's inequality in establishing the energy bound; see Remark \ref{d=5} and also the proof of Proposition \ref{D-bound}.
Moreover, by setting $\phi = 0$ in Theorem \ref{thm-r-i}, we recover the almost sure global well-posedness of the HNLW \eqref{HNLW} in dimension $d = 5$, which is analogous to the result in \cite{OP16}.

 (ii)
 Theorem \ref{thm-r-i} (ii), on $\R^5$, 
 however, holds only for $s > \frac12$. 
 This regularity loss appears in establishing a uniform probabilistic energy bound for approximating random
solutions (Proposition \ref{D-bound}). At this point, we do not know how to close this regularity gap.



\end{remark}

\begin{remark}
\rm

As in discussed in Remark \ref{remark-Td},
the regularity of $\phi$ comes from 
the time-space integrability
issue of the stochastic convolution.
For the case of $\T^d$,
we expect to be able to take $\phi\in\HS(L^2(\T^d);H^s(\T^d))$ 
even for $s<0$; see \cite[Theorem 1.2]{BLL23}.

\end{remark}

\begin{remark}
\rm

To be clear about the adaptedness
to filtrations of the constructed stochastic solutions:
since we have shown that our solution map  \eqref{Duhamel}
is a contraction, the Picard iteration can be used to construct the solution, and
it converges almost surely to the same fixed point (by the uniqueness). 
Each Picard iterate is adapted to the filtration generated by
the noise. 
Therefore, 
the almost sure convergence of the Picard iteration implies that the limiting solution is also adapted. 
See \cite[Remark 1.2]{OO20} for a similar discussion.
Moreover,
the random time is now defined as an entry time of the solution,
which is adapted to the filtration. 
Hence, it qualifies as a stopping time; see page 43 of \cite[(4.5) Proposition, Chapter 1]{RY99}.

\end{remark}

The rest of the paper is organized as follows. In Section \ref{Preliminaries}, 
we introduce some notations,
state regularity properties of the stochastic convolution,
recall the Wiener randomization,
and some probabilistic estimates.
We prove the local well-posedness and establish the stability results in Section \ref{Sec:Local}. 
In Section \ref{proof of th1}, we establish the global well-posedness results for Theorems \ref{thm-d-i} and \ref{thm-r-i}.

\section{Preliminaries}\label{Preliminaries}

In this section, we introduce some notations and go over some preliminary results.

\subsection{Notion, function spaces, and  preliminary results}
In the following, we denote that $A\lesssim B$ if $A\leq CB$ for some constants $C>0$.
We also write $A\ll B$ if the implicit constant should be regarded as small.

Given $ 0< q, r\le \infty$ and a time interval $I\subseteq \R$, 
we consider the mixed Lebesgue spaces $L^q(I;L^r(\R^d))$ of space-time functions $u(t,x) $.
We use notations such as
\[L^q_tL_x^r (I\times \R^d)=L^q(I;L^r(\R^d)),
\quad\text{and}
\quad
L^r_{t,x}(I\times \R^d)
=L^r(I;L^r(\R^d)) , \text{ when }  q=r. 
\]

\noi
In the computation, we may even write 
$L^q_tL_x^r (I\times \R^d)=
L^q_I L^r_x
$
as shorthand notations for the mixed Lebesgue space.

For $s\in\R$ and $r\geq 1$, we denote the homogeneous and inhomogeneous Sobolev norms as

\begin{equation*}
\|f\|_{\dot W^{s,r}(\R^d)}:=\||\nabla|^s f\|_{L^r(\R^d)},
\quad
\|f\|_{W^{s,r}(\R^d)}
:=\|\langle \nabla\rangle^s f\|_{L^r(\R^d)},   
\end{equation*}
and write 
$\dot H^{s}(\R^d)=\dot W^{s,2}(\R^d)$ and $ H^{s}(\R^d)= W^{s,2}(\R^d)$, repectively.
Here $\langle\nabla\rangle$ 
and $|\nabla|$
are defined by 
$\mathcal F{(\langle\nabla\rangle f)}(\xi)
=(1+|\xi|^2)^\frac12 \ft f(\xi)$
and 
$\mathcal F{(|\nabla| f)}(\xi)
=(1+|\xi|^2)^\frac12 \ft f(\xi)$,
respectively.

The energy space is defined as the product norm of the Sobolev space and is denoted as follows:
\begin{equation*}
\dot{\mathcal{H}}^1(\R^d):= \dot H^1(\R^d)\times L^2(\R^d),
\quad
\mathcal{H}^1(\R^d):= H^1(\R^d)\times L^2(\R^d).
\end{equation*}

Let $\varphi\in C_c^\infty(\R^d)$ supported on the ball $|\xi|\leq 2$ and equal to $1$ on the ball $|\xi|\leq 1$. For $N=2^k$, $k\in\Z$, we define the Littlewood-Paley projection operators by
\begin{align*}
&\ft{P_{\leq N}f}(\xi):=\varphi(\xi/N)\ft f(\xi),\\
&\ft{P_{N}f}(\xi):=(\varphi(\xi/N)-\varphi(2\xi/N))\ft f(\xi),\\
&\ft{P_{>N}f}(\xi):=(1-\varphi(\xi/N))\ft f(\xi).
\end{align*}




We recall the definition of the
wave admissible and the Strichartz estimate from
\cite{MZZ15} in the following.
See also
\cite{GV95, LS95, KT98}.
We say that $(q,r)$ is a $\dot H^s(\R^d)$-wave admissible pair if $2\leq q\leq\infty$, $2\leq r<\infty$, such that
\begin{equation*}
\frac{2}{q}\leq(d-1)\Big(\frac{1}{2}-\frac{1}{r}\Big),
\quad
\text{and}
\quad\,
\frac{1}{q}+\frac{d}{r}=\frac{d}{2}-s.
\end{equation*}
Then, we have the following Strichartz estimates.

\begin{lemma}
\label{Strichartz}
Let $d\geq 2$ and  $s>0$.
Let $(q,r)$ be a $\dot H^s(\R^d)$-wave admissible pair, and $(q_1,r_1)$ be a $\dot H^{1-s}(\R^d)$-wave admissible. If $u$ solves
\begin{equation*}
\begin{cases}
\pa_t^2 u-\Delta u+G=0\\
(u,\pa_t u)|_{t=0}=(u_0,u_1)\in\dot\H^s(\R^d)
\end{cases}
\end{equation*}
on $I\times\R^d$, where $I\subset \R$ is a time interval containing $0$.
Then, we have
\begin{equation}
\label{Strichartz estimate}
\|u\|_{L_t^qL_x^r (I\times \R^d)}+
\|(u,\pa_tu)\|_{L_t^\infty\dot\H^s (I\times \R^d) }\les
\|(u_0,u_1)\|_{\dot\H^s(\R^d)}+\|G\|_{L_t^{q_1'} L_x^{r_1'} (I\times \R^d)},
\end{equation}

\noi
where 
$q_1'$ and $r_1'$ are H\"older conjugates of $q_1$ and $r_1$, respectively.
\end{lemma}

One can easily verify from 
wave admissible
definition that 
$(3,\frac{6d}{3d-8})$ is $\dot H^1$-wave admissible for $d\geq 5$. 
Thus, for simplicity, we shall denote the Strichartz space:
\begin{equation}
\label{X(I)}
X(I)=X(I\times\R^d):= 
L_t^3  L_x^{\frac{6d}{3d-8}} (I\times  \R^d),
\end{equation}
and we denote $X(T)=X([0,T])$ for $I=[0,T]$.
Similarly, for $3< q<\infty$  we denote 
\begin{align}
  X_1(I):=L^q_tL^\frac{6d}{3d-8}_x(I\times \mathbb R^d) \text{ and } X_1(T)=X_1([0,T]). \label{X_1}
\end{align}
For $d=5$, we denote that
\begin{equation*}
\wt X(I)=X(I)|_{d=5}:= L_t^3L_x^{\frac{30}{7}}(I\times\R^5),
\end{equation*}
and $\wt X(T)=\wt X([0,T])$.
Moreover,
we define the $Y(T)$-norm and $Z(T)$-norm by 
\begin{align}
\label{def_WY}
\begin{aligned}
\|f\|_{Y(T)}
&=
\|f\|_{L_t^{10}L_x^{\frac{10}{3}}([0,T]\times\R^5)}^{10}+\|\langle\nabla\rangle^{s-\delta} f\|_{L_t^\infty L_x^5([0,T]\times\R^5)},\\
\|f\|_{Z(T)}
&=
\|f\|^6_{L_t^6L_x^{\frac{30}{7}}([0,T]\times\R^5)}+\|f\|^{10}_{L_t^\infty L_x^{\frac{10}{3}}([0,T]\times\R^5)}+\|f\|^2_{L_t^\infty L_x^5([0,T]\times\R^5)}.
\end{aligned}
\end{align}

\noi
Here,  $\delta$ is a sufficiently small constant such that $\frac{1}{2}+\delta<s$.

Next, we recall the Hardy-Littlewood-Sobolev inequality from
\cite[Theorem 1.7]{BCD11},  which will be used repeatedly in later sections.

\begin{lemma}[Hardy-Littlewood-Sobolev]
\label{HLS}
Let $\gamma\in(0,d)$ and $1<p<q<\infty$ satisfy
\begin{equation*}
\frac{1}{p}+\frac{\gamma}{d}=1+\frac{1}{q}.
\end{equation*}

\noindent
Then, there exists a constant $C$ such that
\begin{equation}
\label{HLS-f}
\||\cdot|^{-\gamma}\ast f\|_{L^q(\R^d)}\leq C\|f\|_{L^p(\R^d)}.
\end{equation}
\end{lemma}

Finally, we recall the following useful lemma for the proof of Theorem \ref{thm-r-i}. 
The proof of~\eqref{Riesz} can be found in \cite[Theorem 5.9]{LL01}, which is a property of the Riesz potential. 
The proof of~\eqref{visan lem} closely follows that of \cite[Lemma 5.5]{V07}, and we omit the details.


\begin{lemma}\label{Riesz-Cha}
Let $d\geq 5$.
Then, we have
\begin{equation}
\label{Riesz}
\big\||\nabla|^{-\frac{d-4}{2}}(|u|^2)\big\|^2_{L^2(\R^d)}
\simeq \iint_{\R^d\times\R^d}\frac{|u(x)|^2|u(y)|^2}{|x-y|^4} dx dy ,
\end{equation}

and 
\begin{equation}
\label{visan lem}
\||\nabla|^{-\frac{d-4}{4}}f\|_{L^4(\R^d)}^2\lesssim\||\nabla|^{-\frac{d-4}{2}}|f|^2\|_{L^2(\R^d)}.
\end{equation}

\end{lemma}

\subsection{On the stochastic convolution}\label{Subs-sto}

First we recall the definition of stochastic convolution $\Psi$ of \eqref{Psi}. Then,
we review the regularity properties of several stochastic objects, which can be found in 
\cite{OO20, OPW20,  OPW21, CL22}.

Recall the definition of the Hilbert-Schmidt operator. 
Given two separable Hilbert spaces $H$ and $K$, we denote by $\HS(H;K)$ 
the space of the Hilbert-Schmidt operator $\phi$ from $H$ to $K$. 
This space is endowed with the norm:
\begin{equation*}
\|\phi\|_{\HS(H;K)}=\Big(\sum_{k\in\mathbb{N}}\|\phi e_k\|_K^2 \Big)^{\frac{1}{2}},
\end{equation*}
where $\{e_k\}_{k\in\mathbb{N}}$ is an orthonormal basis of $H$.

Let $(\Omega,\mathcal{F},\mathbb{P})$ be a probability space endowed with a filtration $\{\mathcal{F}_t\}_{t\geq 0}$. Fix an orthonormal basis $\{e_k\}_{k\in\mathbb{N}}$ of $L^2(\R^d)$. Then, a cylindrical Wiener process $W$ on $L^2(\R^d)$ is defined by the following random Fourier series:
\begin{equation*}
W(t)=\sum_{k\in\mathbb{N}}\beta_k(t)e_k,
\end{equation*}

\noi
where $\{\beta_k\}_{k\in\mathbb{N}}$ is a family of mutually independent complex-valued Brownian motions so that $\beta_{-k}=\overline{\beta_k}$, $k\in\mathbb{N}.$ In particular, $\beta_0$ is a standard real-valued Brownian motion. In view of the cylindrical Wiener process $W$, we can express the stochastic convolution $\Psi$ in \eqref{Psi} as 
\begin{equation}
\label{Psi'}
\Psi(t)=\int_0^t 
Q(t-t') \phi dW(t')
=\sum_{k\in\mathbb{N}}\int_0^t
Q(t-t')   \phi e_k d\beta_k(t'),
\end{equation}

\noi
where $Q(\cdot)$ is defined in \eqref{S(t)}.
Next, we  state the regularity properties of the stochastic convolution, 
which shows that $\Psi$ is continuous in time and satisfies a so-called ``Strichartz estimate". 
One can find a detailed proof for part~\eqref{Psi-i} in \cite[Proposition 4.20]{DZ14}, and \cite[Remark 2.7]{BLL23} for part \eqref{Psi-ii}. 
The proof for part~\eqref{Psi-iii} is a simple adaptation of the proof of Lemma 2.6 in \cite{BLL23}.
The key ingredient is the {\it Wiener chaos} estimates;
see \cite[Lemma 2.4]{GKO18} and \cite[Proposition 2.4]{TT10} for the proof. 
We will omit the proof.

\begin{lemma}
\label{regularity of Psi}
Let $d\geq 1$, $T>0$ and $\phi\in \HS(L^2(\R^d);H^s(\R^d))$ for some $s\in\mathbb{R}$. Then, there exists some constant  $C=C(T,p)>0$  such that the following statements hold.
\begin{enumerate}[{\rm (i)}]
\item\label{Psi-i} $\Psi\in C([0,T];H^{s+1}(\R^d))$ almost surely. Moreover, for any $1\leq p<\infty$ we have
\begin{equation*}
\mathbb{E}\Big[\sup_{0\leq t\leq T}\|\Psi(t)\|_{H^{s+1}(\R^d)}^p\Big]\leq C \|\phi\|^p_{\HS(L^2;H^s)}.
\end{equation*}
\item\label{Psi-ii} $\pa_t\Psi\in C([0,T];H^s(\R^d))$ almost surely. Moreover, for any $1\leq p<\infty$ we have
\begin{equation*}
\mathbb{E}\Big[\sup_{0\leq t\leq T}\|\pa_t\Psi(t)\|_{H^{s}(\R^d)}^p\Big]\leq C
\|\phi\|^p_{\HS(L^2;H^s)}.
\end{equation*}

\item\label{Psi-iii} 
%
Let $1\leq q<\infty$ and  $d\geq 3$.
If $\sigma \geq \frac{11}{12}$,
then we have 
$\Psi\in L^q_tW^{s+1-\sigma,\frac{6d}{3d-8}}_x([0,T]\times \R^d)$ almost surely. 
Moreover,  under the same condition $\sigma \geq \frac{11}{12}$, 
for any $1\leq p<\infty$,
we have
\begin{equation*}
\mathbb{E}\Big[\|\Psi\|_{L^q_tW^{s+1-\sigma,\frac{6d}{3d-8}}_x([0,T]\times \R^d)}^p\Big]\leq C
\|\phi\|^p_{\HS(L^2;H^s)}.
\end{equation*}

\item\label{Psi-iv} 
Let $d=5$ and $\phi\in\HS(L^2(\R^5);H^1(\R^5))$.
Then, for any $1\leq p<\infty$ we have
\begin{equation*}
\mathbb{E}
\Big[\|\Psi\|^p_{\wt X(T)}+\|\Psi\|^p_{L_T^\infty L_x^5}+\|\pa_t\Psi\|^p_{L_T^{10}L_x^{\frac{10}{3}}}\Big]
\leq C
\|\phi\|^p_{\HS(L^2;H^1)}.
\end{equation*}

\noi
Moreover, there exist some constants $C,c>0$ such that
\begin{equation*}
\mathbb{P}(\{\|\Psi\|_{\wt X(T)}+\|\Psi\|_{L_T^\infty L_x^5}+\|\pa_t\Psi\|_{L_T^{10}L_x^{\frac{10}{3}}}>\lambda\})
\leq C \exp\Big(-\frac{c\lambda^2}{\|\phi\|^2_{\HS(L^2;H^1)}}\Big),
\end{equation*}
where $\wt X(T)=L_t^3 L_x^{\frac{30}{7}}([0,T]\times\R^5)$.
\end{enumerate}  
\end{lemma}
Note that the condition 
 $\sigma \geq \frac{11}{12}$ in Lemma \ref{regularity of Psi} (iii) ensures the existence of 
  $\tilde q\geq 2$
 such that 
$(\tilde q,\frac{6d}{3d-8})$ 
 forms an 
$\dot H^\sigma $-wave admissible pair for $d\geq 3$, 
thereby enabling the application of the Strichartz estimate in the proof.

\subsection{Wiener randomization and probabilistic estimates}
\label{SUB:WR}

In this subsection,
we briefly go over the randomization 
introduced in 
\cite{LM14, BOP15, BOP15'} for completeness. 
Let $\psi\in\mathcal{S}(\R^d)$ be a real-valued function such that ${\rm  supp}\,\psi\subset[-1,1]^d$, $\psi(-\xi)=\cj{\psi(\xi)}$ 
and
\begin{equation*}
\sum_{n\in\mathbb{Z}^d}\psi(\xi-n)\equiv 1\quad{\text {for\ all}}\quad \xi\in\R^d.
\end{equation*}

\noi
The function $u$ on $\R^d$ can be written as follows:
\begin{equation}
\label{undeco}
u=\sum_{n\in\Z^d}\psi(D-n)u,
\end{equation}
where $\psi(D-n)$ denotes the Fourier multiplier operator with symbol $\psi(\cdot-n)$. 


Let $\{g_{n,j}\}_{n\in\Z^d,j=0,1}$
be a sequence of mean zero complex-valued random variables\footnote{ Here, $g_{-n,j}=\overline{g_{n,j}}$ for all $n\in\Z^d$, $j=0,1$. 
In particular, $g_{0,j}$ is real-valued.}
on the probability space $(\Omega,\mathcal{F},\mathbb{P})$ and independent of the wiener process $W$.
Moreover, assume that
 $\{g_{0,j},\Re\, g_{n,j},\Im\, g_{n,j}\}_{n\in\mathcal{I},j=0,1}$
 are independent and endowed with probability distributions
  $\mu_{0,j}$, $\mu^{(1)}_{n,j}$ and $\mu^{(2)}_{n,j}$. 
Here the index set $\mathcal{I}$ is defined by
\begin{equation*}
\mathcal{I}:=\bigcup_{k=0}^{d-1}\Z^k\times\Z_+\times\{0\}^{d-k-1},
\end{equation*}
and is such that $\mathbb{Z}^d=\mathcal{I}\cup(-\mathcal{I})\cup\{0\}$.
Then,
given a pair of functions $(u_0$, $u_1)$ on $\R^d$, 
the {\it Wiener randomization} 
 $(u_0^\o$, $u_1^\o)$ of  $(u_0$, $u_1)$  
adapted to the uniform decomposition \eqref{undeco} and random variables $\{g_{n,j}\}_{n\in\Z^d,j=0,1}$
is defined by
\begin{equation*}
(u_0^\omega,u_1^\omega)=\Big(\sum_{n\in\Z^d}g_{n,0}(\omega)\psi(D-n)u_0,\sum_{n\in\Z^d}g_{n,1}(\omega)\psi(D-n)u_1\Big). 
\end{equation*}

\noi
We make the following assumption on the probability distributions  
$\mu_{0,j}$, $\mu^{(1)}_{n,j}$ and $\mu^{(2)}_{n,j}$:
there exists $c>0$ such that\footnote{
One can see that \eqref{dis} satisfies Gaussian random variables with the standard normal distributions 
$d\mu_k(x)=(2\pi)^{-\frac{1}{2}}e^{-\frac{x^2}{2}} dx$,
Bernoulli random variables with the distributions $ d\mu_k(x)=\frac{1}{2}(\delta_{-1}+\delta_1) dx$, 
and any random variables with compactly supported distributions.
}
\begin{equation}
\label{dis}
\int_{\R}e^{\gamma x} d\mu_{0,j}\leq e^{c\gamma^2}\ {\text{and}}\ \int_{\R}e^{\gamma x} d\mu^{(k)}_{n,j}\leq e^{c\gamma^2}
\end{equation}
for all $\gamma\in\R,$ $n\in\mathbb{Z}^d$, $j=0,1$, and $k=1,2$.

\begin{remark}
\label{ramdf}
\rm

We see from \cite[Lemma 2.2]{P17} implies that if $(u_0,u_1)\in \H^s(\R^d)$, 
then the randomization $(u^\omega_0,u^\omega_1)$ is almost surely in $\H^s(\R^d)$. 
Moreover, an adaptation of the argument in \cite[Lemma B.1]{BT08-1} gives that for any $s'>s$, if $(u_0,u_1)\in \H^s(\R^d)\backslash \H^{s'}(\R^d)$, then $(u^\omega_0,u^\omega_1)\notin \H^{s'}(\R^d)$ for almost every $\omega\in\Omega$, i.e., 
\[\mathbb{P}(\omega\in\Omega:(u^\omega_0,u^\omega_1)\notin \H^{s'}(\R^d))=1.\]

\noi
Namely, there is no smoothing based on randomization in terms of differentiability. 
The primary use of this randomization is to achieve better integrability
such that,
if $u_j\in L^2(\R^d),j=0,1$, then the randomized functions $u_j^\omega,j=0,1$ are almost surely in $L^p(\R^d)$ for any $p\geq 2$. 
A similar discussion can be found in \cite{P17}, and the proof is provided in \cite[Lemma 4]{BOP15}.

\end{remark}

We also recall the following probabilistic Strichartz estimates.
Define $\wt S(t)$ by
\begin{equation}
\label{t-S(t)}
\wt S(t)(u_0,u_1):=-\frac{|\nabla|}{\langle\nabla\rangle}\sin(t|\nabla|)u_0+\frac{\cos(t|\nabla|)}{\langle\nabla\rangle}u_1,
\end{equation}
which satisfies
\begin{equation}
\label{S(t) t-S(t)}
\pa_t S(t)(u_0,u_1)=\langle\nabla\rangle\wt S(t)(u_0,u_1).
\end{equation}

\begin{lemma}
\label{P-est-S*}
Let $d\geq 3$.
Given a pair  of real-valued functions 
$(u_0,u_1)$ defined on $\R^d$,
let $(u_0^\omega,u_1^\omega)$ be the randomization defined in \eqref{R (u0,u1)},
 satisfying \eqref{dis}.
 Let $T>0$ and $S^\ast(t)=S(t)$ or $\wt S(t)$
defined in \eqref{S(t)} and \eqref{t-S(t)}, respectively. Then, we have
\begin{enumerate}[{\rm (i)}]
\item\label{rand(i)} If $(u_0,u_1)\in\H^0(\R^d)$, then given $1\leq q<\infty$ and $2\leq r<\infty$, there exist $C,c>0$ such that
\begin{align}
\label{R est1}
\begin{aligned}
\mathbb{P}(\big\|S^\ast(t)(u^\omega_0,u^\omega_1) 
\big\|_{L_t^qL_x^r([0,T]\times\R^d)}>\lambda)
\leq C \exp  \Big(-\frac{c  \lambda^2}{\max\{1,T^2\}T^{\frac{2}{q}}\|(u_0,u_1)\|_{\H^0}^2} \Big).
\end{aligned}
\end{align}

\item\label{rand(ii)} For any $\e>0$, if $(u_0,u_1)\in\H^\e(\R^d)$, then given $2\leq r\leq \infty$, there exists $C, c>0$ such that
\begin{align}
\label{R est2}
\begin{aligned}
\mathbb{P}(\big\|S^\ast(t)(u^\omega_0,u^\omega_1)
\big\|_{L_t^\infty L_x^r([0,T]\times\R^d)}>\lambda)
\leq
C(1+T)\exp \Big(-\frac{c\lambda^2}{\max(1,T^2)\|(u_0,u_1)\|_{\H^\e}^2}\Big).
\end{aligned}
\end{align}

\end{enumerate}
\end{lemma}

The proof of \eqref{R est1} can be found in \cite[Proposition 2.3]{P17},
and a straightforward adaptation also gives
the proof of $\wt S(t)$.
As for part \eqref{R est2}, one can refer to \cite[Proposition 3.3]{OP16} and the references therein, 
where one can be readily extended to the cases of $d\geq4$.

\begin{remark}
\rm

The probabilistic Strichartz estimates in Lemma \ref{P-est-S*} are associated with the finite time interval $[0,T]$
indicating that obtaining a global-in-time Strichartz estimate directly is not possible. 
Recently, Dodson-Luhrmann-Mendelson \cite{DLM20} achieved global well-posedness 
and scattering for the $4$-dimensional energy-critical stochastic NLW
by using a radial Sobolev estimate and a global-in-time probabilistic Strichartz estimate. 
Due to the limitation of the local-in-time probabilistic Strichartz estimate in our paper, 
scattering results are not presented. 
Establishing a global-in-time probabilistic Strichartz estimate
for $d=5$ 
and obtaining almost sure scattering results for (deterministic) HNLW is a topic we hope to address in the future.

\end{remark}

\section{Local well-posedness and stability theories}
\label{Sec:Local}
In this section, we consider the
following 
defocusing energy-critical HNLW equation with a perturbation:
\begin{equation}
\label{local v}
\begin{cases}
\pa_t^2 v-\Delta v+\mathcal{N}(v+f)=0\\
(v,\pa_t v)|_{t=t_0}=(v_0,v_1)
\end{cases}
\end{equation}

\noindent
with $\mathcal{N}(v)=(|\cdot|^{-4}\ast v^2)v$.
Here, $f$ is a given deterministic function satisfying certain regularity conditions.

\subsection{Local well-posedness of the perturbed  HNLW}

By employing of the contraction mapping theorem in conjunction with H\"{o}lder's inequality, the Hardy-Littlewood-Sobolev inequality, and the triangle inequality, we establish the following local well-posedness of equation \eqref{local v}. 

\begin{proposition}\label{local}
Let $d\geq 5$ and $(v_0,v_1)\in\dot\H^1(\R^d)$. There exists $0<\eta\ll 1$ such that  if
\begin{equation}
\label{local assumption1}
\|S(t-t_0)(v_0,v_1)\|_{X(I)}\leq\eta\quad\text{and}\quad\|f\|_{X(I)}\leq\eta
\end{equation}

\noi
for some interval $I=[t_0,t_1]\subset\R$, then the Cauchy problem \eqref{local v} admits a unique solution 
$(v,\pa_t v)\in C(I;\dot\H^1(\R^d))$ satisfies
\begin{equation*}
\|v\|_{X(I)}\leq 3\eta.
\end{equation*}

\noi
In particular, the uniqueness of $v$ holds in the ball
\begin{equation}
\label{ball}
B_{\eta}(I)
:=
\{v\in X(I):\|v\|_{X(I)}\leq3\eta\}.  
\end{equation}

\end{proposition}

\begin{proof}
We define the solution map $\Phi$ of \eqref{local v} by
\begin{equation*}
\Phi(v):=
S(t-t_0)(v_0,v_1)-\int_{t_0}^t  Q(t-t')\mathcal{N}(v+f)(t')dt'.
\end{equation*}
Then, we show $\Phi$ is a contraction on the ball $B_\eta(I)$, and $\eta>0$ is to be chosen later.

By Lemma \ref{Strichartz}, we have
\begin{equation*}
\|\Phi(v)\|_{X(I)}\leq
\|S(t-t_0)(v_0,v_1)\|_{X(I)}+C_1\|\mathcal{N}(v+f)\|_{L_I^1L_x^2}
\end{equation*}
for some constant $C_1>0$.
Then, we apply H\"{o}lder's inequality and the Hardy-Littlewood-Sobolev inequality \eqref{HLS-f} to obtain that
\begin{align}
\begin{aligned}
\|\mathcal{N}(v+f)\|_{L_I^1L_x^2}
 &=\|(|\cdot|^{-4}\ast|v+f|^2)(v+f)\|_{L_I^1L_x^2}\\
 &\leq
 \||\cdot|^{-4}\ast|v+f|^2\|_{L_I^{3/2}L_x^{3d/4}}\|v+f\|_{X(I)}\\
&\leq
C_2\||v+f|^2\|_{L_I^{3/2}L_x^{3d/(3d-8)}}\|v+f\|_{X(I)}
\leq
C_2\|v+f\|_{X(I)}^3
\end{aligned}
\label{local estimate1}
\end{align}

\noi
for some constant $C_2>0$.
By the triangle inequality, assumption \eqref{local assumption1}, definition \eqref{ball}, and estimate \eqref{local estimate1}, we deduce
\begin{equation*}
\begin{split}
\|\Phi(v)\|_{X(I)}
\leq\eta+C_3\|v\|^3_{X(I)}+C_3\|f\|^3_{X(I)}
\leq\eta+C_3(3\eta)^3+C_3\eta^3
\leq 3\eta,
\end{split}
\end{equation*}

\noi
where $C_3>0$ is a constant and $0<\eta\ll 1$ such that $C_3(3\eta)^3\leq\eta$. From this, we prove that $\Phi$ maps $B_{\eta}(I)$ into itself. Repeating the above steps, for $v_1,v_2\in B_{\eta}(I)$, we have
\begin{equation}\label{local est1}
\begin{split}
\|\Phi(v_1)-\Phi(v_2)\|_{X(I)}
 &\leq C_1\|\mathcal{N}(v_1+f)-\mathcal{N}(v_2+f)\|_{L_I^1L_x^2}\\
&\leq C_1\|(|\cdot|^{-4}\ast|v_1+f|^2)(v_1-v_2)\|_{L_I^1L_x^2}\\
&\quad +C_1\|(|\cdot|^{-4}\ast|v_1-v_2|^2)(v_2+f)\|_{L_I^1L_x^2}\\
&\quad +2C_1\|(|\cdot|^{-4}\ast(v_2+f)(v_1-v_2))(v_2+f)\|_{L_I^1L_x^2}\\
&\eqqcolon \1 + \II + \III,
\end{split}
\end{equation}

\noindent where we used the fact that
 \begin{equation*}
 \begin{split}
 \mathcal{N}(v_1+f)-\mathcal{N}(v_2+f)
& =(|\cdot|^{-4}\ast|v_1+f|^2)(v_1-v_2)+(|\cdot|^{-4}\ast|v_1-v_2|^2)(v_2+f)\\
 &\quad +2(|\cdot|^{-4}\ast(v_2+f)(v_1-v_2))(v_2+f).
 \end{split}
 \end{equation*}

\noi
For the first term $\1$, by using H\"{o}lder's, Hardy-Littlewood-Sobolev, and  triangle inequalities,
we yield
\begin{equation}\label{local est2}
\begin{split}
\1
\lesssim
\|v_1-v_2\|_{X(I)}\|v_1+f\|_{X(I)}^2
\lesssim
\eta^2\|v_1-v_2\|_{X(I)}.
\end{split}
\end{equation}

\noi
Similarly, we have
\begin{align}\label{local est3'}
    \II\lesssim\eta^2\|v_1-v_2\|_{X(I)},
    \quad
    \III\lesssim\eta^2\|v_1-v_2\|_{X(I)}.
\end{align}

%
%
%
%
%

\noi
Therefore, from
estimates \eqref{local est1}, \eqref{local est2},
and \eqref{local est3'}, we have
\begin{equation*}
\|\Phi(v_1)-\Phi(v_2)\|_{X(I)}\leq C\eta^2\|v_1-v_2\|_{X(I)}
\end{equation*}
for some constant $C>0$. By choosing $\eta$ such that
$C\eta^2\leq \frac{1}{2}$,
we can conclude that $\Phi$ is a contraction on $B_{\eta}(I)$.

Finally, we show that $(v,\pa_t v)\in L_I^\infty\dot\H^1$. An application of the Strichartz estimate, H\"{o}lder's inequality and the Hardy-Littlewood-Sobolev inequality as well as the following Duhamel formula
\begin{equation*}
v(t)=S(t-t_0)(v_0,v_1)-\int_{t_0}^t Q(t-t')\mathcal{N}(v+f)(t') dt' 
\end{equation*}

\noi
yields that
\begin{equation*}
\|(v,\pa_t v)\|_{L_I^\infty\dot\H^1}\leq 
\|(v_0,v_1)\|_{\dot\H^1}+3\eta.
\end{equation*}
This completes the proof of this proposition.

\end{proof}

The following blow-up criterion follows as a direct consequence of Proposition \ref{local}.

\begin{proposition}
\label{blow-up}
Let $d\geq 5$ and 
$(v_0,v_1)\in\dot\H^1(\R^d)$.
Let  $f$ be a function on $[0,T]\times \R^d$
satisfying $\|f\|_{X(T)}<\infty$ for any $T>0$. 
If $v$ is a solution to the equation \eqref{local v},
with maximal time existence interval $[0,T^*]$.
Then,  either 
\[
T^*=\infty
\quad \text{ or } 
\quad
\lim_{T\to T^*} \|v\|_{X(T)}=\infty,
\]

\noi
where 
$T^*=T^*(u_0, u_1, f)>0$ denotes  the forward maximal time of existence.
\end{proposition}

\begin{remark}
\label{RM:local}
\rm 

 We note that the local well-posedness of equations \eqref{perturbed v} and \eqref{eq_SHW3}
 follows from the regularity properties of the stochastic convolution $\Psi$ and $\wt \Psi$, as established in Proposition~\ref{local}.

 (i) Let $\phi\in \HS(L^2(\R^d);L^2(\R^d))$.
By replacing $f=\Psi$ 
in  Proposition~\ref{local},
 and using Lemma~\ref{regularity of Psi}, we conclude that equation \eqref{perturbed v}
is almost surely locally well-posed in the energy space $\dot \H^1(\R^d)$.

(ii)
Let $d= 5$ and $s \in [0,1)$.
Assume 
$(u_0,u_1)\in \mathcal{\dot  H}^s(\R^5)$
and $\phi\in\HS(L^2(\R^5);L^2(\R^5))$. Then by
Lemmas 
 \ref{regularity of Psi} 
 and \ref{P-est-S*} 
together with the Sobolev embedding, we obtain 
that 
\begin{align}
\label{regpsi2}
\wt\Psi\in  X(T),
\end{align}
where 
$T=T_\omega$ is almost surely positive. Here 
$\wt \Psi $  is defined in \eqref{Psi2} and $ X(T)$ is defined  in \eqref{X(I)}, respectively.
Hence,  Proposition \ref{local} also implies 
 local well-posedness of
\eqref{eq_SHW3} by replacing $f=\wt \Psi$ and $(v_0,v_1)=(0,0)$,
consequently establishing Theorem \ref{thm-r-i} (i).
See \cite[Proposition 4.3]{P17} and \cite[Lemma 5.1]{OP16}.

\rm
\end{remark}

\subsection{Stability results}

In this subsection, we will state 
 the perturbation lemma, 
 which will be crucial in proving global existence for \eqref{local v}.
We consider
\begin{equation}
\label{HNLW3}
\begin{cases}
\pa_t^2 u-\Delta u+\mathcal{N}(u)=0\\
(u,\pa_t u)|_{t=t_0}=(u_0,u_1)
\end{cases}
\end{equation}

\noi
with $\mathcal{N}(u)=(|\cdot|^{-4}\ast|u|^2)u$.

\begin{lemma}
\label{perturbation}
Let $d\geq 5$,
 $t_0\in  I\subset\R$ be a compact time interval, and $M>0$. 
 Let $v$ be a function on $I\times\R^d$,
 which solves \eqref{HNLW3} with a  perturbation  term $e$,
namely,
\begin{equation*}
\pa_t^2 v-\Delta v+\mathcal{N}(v)=e,
\end{equation*}

\noi
with initial data $(v,\pa_t v)|_{t=t_0}=(v_0,v_1)\in\dot\H^1(\R^d)$, satisfying
\begin{equation}
\label{p-condition1}
\|v\|_{X(I)}\leq M.
\end{equation}

\noi
Let $w$ be the solution to \eqref{HNLW3} with initial data 
 $(w,\pa_t w)|_{t=t_0}=(w_0,w_1)\in\dot\H^1(\R^d)$.
Then, there exists $\wt\e=\wt\e(M)>0$ sufficiently small 
such that if 
\begin{align}
\label{p-condition2}
\begin{aligned}
\|(v_0-w_0,v_1-w_1)\|_{\dot\H^1}&\leq\e,\\
\|e\|_{L_I^1L_x^2}&\leq\e
\end{aligned}
\end{align}

\noi
for some
$0<\e<\wt\e$,
then the following holds:
\begin{equation*}
\sup_{t\in I}\|(v(t)-w(t),\pa_t v(t)-\pa_t w(t))\|_{\dot\H^1_x}+\|v-w\|_{L_I^q L_x^r}\leq C(M)\e
\end{equation*}
for all $\dot H^1$-wave admissible pairs $(q,r)$,
and  $C(\cdot)$ is a non-decreasing function.
\end{lemma}

The proof of the above perturbed lemma is a simple adaption of the proof of  in \cite[Lemma~4.4]{P17}. 
One can also find a similar proof in \cite[Lemma 2.5]{MZZ14}.

Next, we have the following global space-time (Strichartz) bound for the solution to the defocusing energy-critical HNLW;
see proof in \cite{MZZ14}.

\begin{lemma}\label{st-bound}
Let $d\geq 5$ and $w$ be the solution to 
\eqref{HNLW3} with initial data
 $(w,\pa_t w)|_{t=t_0}=(w_0,w_1)\in\dot\H^1(\R^d)$. 
Then, we have
\begin{equation*}
\|w\|_{X(\R_+)}\leq C(\|(w_0,w_1)\|_{\dot\H^1(\R^d)}).
\end{equation*}
\end{lemma}

 At the end of this section, 
 we present the following proposition, 
 which is a simple adaptation of the proof in \cite[Proposition 5.2]{OP16}.
 
\begin{proposition}\label{condition LWP}
Given $f\in \wt X([0,T])$ for $T>0$,
and dyadic $N\geq 1$. 
Let $f_N=P_{\leq N}f$ and $v_N=P_{\leq N}v$ such that $v_N$ be a solution 
to \eqref{local v} on $[0,T]\times\R^5$
with perturbation $f_N$ and 
initial data $(v_N,\pa_t v_N)|_{t=0}=(0,0)$.
Assume that the following conditions hold:
\smallskip
\begin{enumerate}[{\rm (i)}]
\item There exists $K,\theta>0,\tau_\ast>0$ such that
\begin{equation*}
\|f\|_{\wt X(I_j)}\leq K|I_j|^{\theta}\ll 1,
\end{equation*}
where $I_j=[j\tau_\ast,(j+1)\tau_\ast]\cap[0,T]\subset[0,T],$ $j=0,\cdots,[\frac{T}{\tau_\ast}]$.

\medskip
\item There exists $N_0=N_0(T,C_0(T))\gg 1$ such that for $N\geq N_0$, there exists a solution $v_N$ to \eqref{local v}
and satisfies
\begin{equation*}
\sup_N\sup_{t\in[0,T]}\|(v_N,\pa_t v_N)\|_{\H^1(\R^5)}\leq C_0(T)<\infty.
\end{equation*}

\medskip
\item 
There exists $\alpha>0$ and $N_0\gg 1$ such that
for $N\geq N_0\gg1$
\begin{equation*}
\|f-f_N\|_{\wt X([0,T])}<C_1(T)N^{-\alpha}.
\end{equation*}
\end{enumerate}

\smallskip
\noi
Then, there exists a unique solution $(v,\pa_tv)\in C([0,T];\H^1(\R^5))$ to \eqref{local v} with initial data $(v,\pa_t v)|_{t=0}=(0,0)$, satisfying
\begin{equation*}
\sup_{t\in[0,T]}\|(v,\pa_t v)\|_{\H^1(\R^5)}\leq 2C_0(T)<\infty.
\end{equation*}
\end{proposition}

\section{Almost surely global well-posedness
of stochastic HNLW
}
\label{proof of th1}

In this section, 
we present the proofs of Theorems \ref{thm-d-i} and \ref{thm-r-i}.
For both proofs, getting a control on the energy is the main issue.
We first treat the situation of deterministic initial data, 
where  energy growth, and the particular non-conservation source
 is due to the stochastic forces.
By using Ito's lemma and the Burkholder-Davis-Gundy inequality 
one can obtain 
an a priori control for the desire energy bound; 
see Proposition \ref{energy bound}.
This approach has been successfully applied to 
other mass/energy-critical stochastic dispersive PDEs
\cite{OO20, CL22, BLL23}.
On the other hand,
when we consider the random initial data,
the energy is never conserved even if we drop the stochastic forces,
this is due to our data are below the energy space.
We construct  a uniform probabilistic energy bound 
for approximating random solution;
see Proposition \ref{energy-b} below.
As we mentioned earlier, this approach was applied to NLW in \cite{BT08-1, P17}.
However, we need a more 
 intricate analysis to prove a probabilistic
energy bound as in \cite{OP16}.
Finally,
by combining 
an a priori bound for the energy
as well as tools from
the previous sections, 
we prove global existence by an iterative application of the perturbation lemma.
This method is rather standard, based on the previous works.

\subsection{Energy bound with deterministic initial data}

In this subsection, 
we present the energy bound of \eqref{eq_SHW} 
with deterministic initial data. 
To show the energy of \eqref{eq_SHW} is almost surely bounded, we use the following truncated Cauchy problem as an intermediary:
\begin{equation}
\label{truncated SNLW}
\begin{cases}
\pa_t^2 u_N-\Delta u_N+P_{\leq N}\mathcal{N}(u_N)=P_{\leq N}\phi\xi\\
(u_N,\pa_t u_N)|_{t=0}=(u_{0,N},u_{1,N}),
\end{cases}
\end{equation}

\noi
where $u_{j,N}=P_{\leq N} u_{j}$ for $j=0,1$, $\mathcal{N}(u_N)=(|\cdot|^{-4}\ast|u_N|^2)u_N$ and $N\in\mathbb{N}$.

The following lemma shows 
that solution  $u_{\leq N}$ of \eqref{truncated SNLW} 
converges 
to the solution $u$ of \eqref{eq_SHW}. 
The proof is 
a straightforward adaptation of \cite[Lemma 4.1]{BLL23}, utilizing the Hardy-Littlewood-Sobolev inequality.
\begin{lemma}
\label{u_N conv}
Let $d\geq 5$, $\phi\in \HS(L^2(\R^d);L^2(\R^d))$ and $(u_0,u_1)\in\dot\H^1(\R^d).$
Then, the following holds true almost surely:
Assume that $u$ is a solution to \eqref{eq_SHW}
on $[0,T]$ for some $T>0$,
and assume that $u_{\leq N}$ 
is a solution to the truncated equation \eqref{truncated SNLW} on $[0,T_N]$ for some $T_N>0$. 
Also, let $R>0$ be such that $\|u\|_{X(T)}\leq R$,
where the $X(T)$-norm is defined in \eqref{X(I)}.
Then, by letting $S_N=\min(T,T_N)$, 
we have the following, as $N\to\infty$,
\begin{equation*}
\|(u-u_N,\pa_t u-\pa_t u_N)\|_{C([0,S_N];\dot\H^1(\R^d))}
\longrightarrow 0
\end{equation*}

\noi
and
\begin{equation*}
\|u-u_N\|_{X(S_N)}\longrightarrow 0.
\end{equation*}
\end{lemma}

Next, we show that the energy of \eqref{eq_SHW} can be  controlled almost surely.
\begin{proposition}
\label{energy bound}
Let $d\geq 5$, $\phi\in \HS(L^2(\R^d);L^2(\R^d))$, and 
$(u_0,u_1)\in\dot\H^1(\R^d)$.
Let~$u$ be the solution to \eqref{eq_SHW} with $(u,\pa_t u)|_{t=0}=(u_0,u_1)$
and let $T^\ast=T^\ast(\omega,u_0,u_1)$ be the forward maximal time of existence.
Then, given any $T_0>0$,
there exists $C=C(\|(u_0,u_1)\|_{\dot\H^1},\|\phi\|_{\HS(L^2;L^2)},T_0)>0$
 such that for any stopping time $T$ with $0<T<\min(T^\ast,T_0)$ almost surely, we have
\begin{equation}
\label{energy bounded}
\mathbb{E}\Big[\sup_{0\leq t\leq T}E( u(t),\pa_t u(t))\Big]\leq C(u_0, u_1, \phi, T_0).   
\end{equation}

\end{proposition}

\begin{proof}

Let us first recall the
energy functional $E(u,\dt u)$ from \eqref{energy}, and then we 
rewrite $E(u,\dt u)$ as $E(
u_1,u_2)=E(u,\dt u)$ such that 
\begin{equation*}
\begin{split}
E[\cj u]:=E(
u_1,u_2)=\frac{1}{2}\int_{\R^d}|u_2|^2 dx+\frac{1}{2}\int_{\R^d}|\nabla u_1|^2 dx
+\frac{1}{4}\iint_{\R^d\times\R^d}\frac{|u_1(x)|^2|u_1(y)|^2}{|x-y|^4} dx dy.    
\end{split}
\end{equation*}

\noi
Now, we 
take the functional derivative of $E[\cj u]$ (denoted $\dl E/\dl \cj u$),
and arrive the following:
\begin{equation*}
\frac{\dl E}{\dl \cj u}[\cj u; \cj v]
=\int_{\R^d}
(u_2v_2+\nabla u_1\cdot\nabla v_1) dx+\iint_{\R^d\times\R^d}\frac{u_1(x)v_1(x)|u_1(y)|^2}{|x-y|^4} dx dy,
\end{equation*}

\noi
where $\cj v=(v_1, v_2)$ is the
test function.
The second functional derivative of $E[\cj u]$ is then
\begin{equation*}
\begin{split}
\frac{\dl^2 E}{ (\dl \cj u)^2  }
[\cj u; \cj v;\cj w]
=&\int_{\R^d} 
v_2w_2  + \nabla v_1\cdot\nabla w_1 dx+\iint_{\R^d\times\R^d}\frac{v_1(x)w_1(x)|u_1(y)|^2}{|x-y|^4} dx dy\\
&+2
\iint_{\R^d\times\R^d}\frac{u_1(x)u_1(y)v_1(x)w_1(y)}{|x-y|^4} dx dy
\end{split}
\end{equation*}

\noi
with $\cj w:=(w_1,w_2)$.
Given $R>0$, we define the stopping time
\begin{equation*}
T_1=T_1(R):=\inf\{\tau>0,\|u\|_{X([0,\tau])}\geq R\}, 
\end{equation*}

\noi
where $X([0,\tau])$-norm is defined in \eqref{X(I)}. Set $T_2:=\min\{T,T_1\}$. Note that
\begin{equation*}
\|u\|_{X(T)}<\infty
\end{equation*}
almost surely in view of the blowup alternative in Proposition \ref{blow-up}, so that we have $T_2\nearrow T$ almost
surely as $R\to\infty$.

Note that, we can write \eqref{truncated SNLW} in the following Ito formulation
\begin{equation*}
\begin{cases}
d
\Big(\begin{matrix}
u_N\\
\pa_t u_N
\end{matrix}\Big)
=\Big\{\Big(\begin{matrix}
0&1\\
\Delta&0
\end{matrix}\Big)
\Big(\begin{matrix}
u_N\\
\pa_t u_N
\end{matrix}\Big)
+
\Big(\begin{matrix}
0\\
-P_{\leq N}\mathcal{N}(u_N)
\end{matrix}\Big)
\Big\}dt+\Big(\begin{matrix}
0\\
P_{\leq N}\phi dW
\end{matrix}\Big)\\
~\\
(u_N,\pa_t u_N)|_{t=0}=(u_{0,N},u_{1,N}).
\end{cases}
\end{equation*}

\noi
By using Ito's lemma (\cite[Theorem 4.32]{DZ14}), 
we derive the following for $0<t<T_2$,
\begin{align*}
E(u_N(t),\pa_t u_N(t))
&=E(u_{0,N},u_{1,N})+2t\|P_{\leq N}\phi\|_{\HS(L^2;L^2)}^2 \\
&\quad +\sum_{k\in\mathbb{N}}\int_0^t\int_{\R^d}\pa_t u_N(t')P_{\leq N}\phi e_k dx d\beta_k(t') \\
&\quad +\int_0^t\int_{\R^d}\pa_t u_N(t')({\rm Id}-P_{\leq N})\mathcal{N}(u_N) dx dt'\\
&=:
E(u_{0,N},u_{1,N})+2t\|P_{\leq N}\phi\|_{\HS(L^2;L^2)}^2 
+\1+\II.
\end{align*}

For the term $\1$, 
we use Burkholder-Davis-Gundy, H\"{o}lder's,
  and Cauchy-Schwarz inequalities
to obtain 
\begin{equation}
\label{energy est1'}
\begin{split}
\mathbb{E}\Big[\sup_{0\leq t\leq T_2}   |\1 | \Big]
& \leq 
C\mathbb{E}\Big[\Big(\sum_{k\in\mathbb{N}}\int_0^{T_2}\Big|\int_{\R^d}\pa_t u_N(t')P_{\leq N}\phi e_k dx\Big|^2 dt'\Big)^\frac{1}{2}\Big]\\
& \leq 
CT_2^{\frac{1}{2}}\mathbb{E}\Big[\|\pa_tu_N\|_{L_t^\infty L_x^2}\|\phi\|_{\HS(L^2;L^2)}\Big]\\
& \leq 
\frac{C}{2}T_2\|\phi\|_{\HS(L^2;L^2)}^2+\frac{1}{2}\mathbb{E}\Big[\|\pa_tu_N\|^2_{L_t^\infty L_x^2}\Big]. 
\end{split}
\end{equation}

As for $\II$, 
by H\"{o}lder's, Hardy-Littlewood-Sobolev inequalities,
and apply Lebesgue dominated convergence theorem to 
$({\rm Id}-P_{\leq N})\mathcal{N}(u)$ and Lemma \ref{u_N conv}, for $t\in[0,T_2]$
we have
\begin{equation}
\label{energy est2'}
\begin{split}
\II &\lesssim
\|\pa_t u_N\|_{L_t^\infty L_x^2}\|({\rm Id}-P_{\leq N})\mathcal{N}(u_N)\|_{L_t^1L_x^2}\\
&\lesssim
\|\pa_t u_N\|_{L_t^\infty L_x^2}\Big(\|({\rm Id}-P_{\leq N})\mathcal{N}(u)\|_{L_t^1L_x^2}
+\|\mathcal{N}(u)-\mathcal{N}(u_N)\|_{L_t^1L_x^2}\Big)\\
&\lesssim
\|\pa_t u_N\|_{L_t^\infty L_x^2}\Big(\|    P_{> N} \mathcal{N}(u)\|_{L_t^1L_x^2}
 +(\|u\|_{X(T_2)}^2+\|u_N\|_{X(T_2)}^2)\|u-u_N\|_{X(T_2)} \Big)\\
&\quad\longrightarrow 0,
\end{split}
\end{equation}
as $N\to\infty.$
\noi
Therefore,
from estimates \eqref{energy est1'} and \eqref{energy est2'}
 we can conclude 
\begin{equation*}
\begin{split}
\mathbb{E}
\Big[\sup_{0\leq t\leq T_2}E(u_N,\pa_t u_N)\Big] &\leq
\frac{1}{2}\mathbb{E}\Big[\sup_{0\leq t\leq T_2}E(u_N,\pa_t u_N)\Big]
+CE(u_0,u_1)+CT_2\|\phi\|_{\HS(L^2;L^2)}^2+\e,
\end{split}
\end{equation*}

\noi
where $\e>0$ is arbitrarily small and $N\geq N_0\eqqcolon N_0(\e,u)$ sufficiently large. Hence, we have
\begin{equation*}
\mathbb{E}\Big[\sup_{0\leq t\leq T_2}E(u_N,\pa_t u_N)\Big]
\leq
CE(u_0,u_1)+CT_2\|\phi\|_{\HS(L^2;L^2)}^2+\e.    
\end{equation*}
By Fatou's lemma, we have
\begin{equation*}
\mathbb{E}\Big[\sup_{0\leq t\leq T_2}E(u,\pa_t u)\Big]
\leq\liminf_{N\to\infty}\mathbb{E}\Big[\sup_{0\leq t\leq T_2}E(u_N,\pa_t u_N)\Big]\leq C
\end{equation*}
for some constant $C=C(\|(u_0,u_1)\|_{\dot\H^1},\|\phi\|_{\HS(L^2;L^2)},T_0)$.  In view of the almost sure convergence of $T_2$ to $T$, the energy bound \eqref{energy bounded} is obtained by using Fatou's lemma.
\end{proof}

\subsection{Proof of Theorem \ref{thm-d-i}}
In this section, we prove Theorem \ref{thm-d-i}. 
The proof relies on  the almost sure energy bound 
established in Proposition \ref{energy bound}, 
stability theories in Section \ref{Sec:Local},
and the global space-time bound provided by Lemma \ref{st-bound}.
See mass/energy-critical NLS studies in \cite{OO20, CL22} and 
energy-critical NLW research in \cite{BLL23}.
For the sake of notation, in this subsection we will write 
\[
E(u, \dt u)(t) = E(u)(t).
\]

The essential difficulty in the energy-critical problem is that an energy bound alone is not sufficient to extend the solution globally in time. In the deterministic setting, one must employ concentration-compactness arguments and other delicate techniques to rule out finite-time blow-up, soliton-like solutions, and low-to-high frequency cascades; see \cite{MZZ14}. 
This necessity arises because, as shown in Proposition \ref{blow-up}, the Strichartz norm of the solution characterizes the blow-up criterion. Therefore, demonstrating the finiteness of the Strichartz norm becomes the main task.

However, the standard deterministic tools for establishing the finiteness of the Strichartz norm (e.g., Lemma \ref{st-bound}) are not available in our stochastic setting. Nevertheless, following the strategy of \cite{OO20}, we identify two key ingredients in the proof of Theorem \ref{thm-d-i}:

\begin{itemize}

\item
To ensure the finiteness of the Strichartz norm, we require that both the energy and the stochastic convolution remain uniformly bounded almost surely. To this end, we construct two stopping times, $\tau_k$ and $\gamma_k$, defined as follows:
\begin{enumerate}
\item $\tau_k$: the minimal time at which the energy exceeds a fixed constant;
\item $\gamma_k$: the minimal time at which the stochastic convolution exceeds a fixed constant.
\end{enumerate}
We then prove that for any given deterministic time $T_0 > 0$, 
it holds that $\mathbb{P}(\tau_k \wedge \gamma_k \geq T_0) = 1$; see Lemmas \ref{LEM:st1} \& \ref{LEM:st2}.

\item
By introducing a globally defined deterministic solution as a reference path (with existence time assumed to be $T_0 > 0$), and invoking the stability lemma (Lemma \ref{perturbation}), we show that the Strichartz norm of the solution in the stochastic setting remains finite up to time $T_0$, almost surely. 
This implies that the solution can be extended up to time $T_0$, almost surely 
(provided that both the energy and the stochastic convolution are uniformly bounded almost surely, as established in the previous point).

\end{itemize}

\begin{lemma}
\label{LEM:st1}

Let $k\geq 1$ and $T_0>0$ be any deterministic value.
We define a sequence of increasing  stopping times $ \tau_k$ to be
\begin{align}
  \tau_k = \inf \{  t \geq 0  \,:\,
\sup_{0< r \leq t }  E(u)(r) \geq   k
\},
\label{stp}
\end{align}

\noi
and we define the sets
\begin{align*}
A_k =\{  \o \,:\,   \tau_k    \geq  T_0  \}
\quad\quad
\text{and}
\quad\quad
A= \bigcup_{k \geq 1} A_k.
\end{align*}

\noi
Then,  we have 
$ \mathbb{P} ( A ) =1$.

\end{lemma}

\begin{proof}

Given $k\geq 1$, any   $T_0>0$,
and $  \tau_k$ to be as in \eqref{stp}.
Then,
from Proposition \ref{energy bound}, we have 
\begin{equation}
\label{energy bounded2}
\mathbb{E}\Big[ \sup_{0\leq r\leq (T_0-1)\wedge   \tau_k  }E( u )(r)  \Big]
\leq C(u_0, u_1, \phi, T_0),   
\end{equation}

\noi
which is uniform in $k$.
Since $\tau_k$ is increasing in $k$, we let 
\begin{align*}
  \tau_\infty = \lim_{k\to \infty}   \tau_k
\end{align*}

\noi
and since energy is a 
non-decreasing function 
in time,
we have
\begin{align}
\lim_{k\to \infty}\sup_{0\leq r\leq (T_0-1)\wedge   \tau_k }E( u )(r) 
=\sup_{0\leq r\leq (T_0-1)\wedge   \tau_\infty  }E( u )(r) .
\label{eglim}
\end{align}

\noi
Hence, by using Fatou’s lemma, and 
from \eqref{energy bounded2}, \eqref{eglim},
 we obtain
\begin{equation*}
\mathbb{E}\Big[ \sup_{0\leq r\leq (T_0-1)\wedge   \tau_\infty  }E( u)(r)  \Big]
\leq C(u_0, u_1, \phi, T_0) <\infty.
\end{equation*}

\noi
Then, we have
\begin{equation*}
\mathbb{P} 
(   \sup_{0\leq r\leq (T_0-1)\wedge   \tau_\infty  }E( u ) (r)    <\infty )
=1.
\end{equation*}

\noi
We note,
\begin{align}
\begin{aligned}
\infty>\sup_{0\leq r\leq (T_0-1)\wedge   \tau_\infty  }E( u  )(r) 
&= \ind_{  \{  \tau_\infty > T_0 -1  \} } 
\sup_{0\leq r\leq (T_0-1)   } 
E( u  ) (r)   \\
&\quad  +
\ind_{ \{   \tau_\infty < T_0 -1 \} } 
\sup_{0\leq r\leq   \tau_\infty   } 
E( u   ) (r) ,
\end{aligned}
\label{indsum}
\end{align}

\noi
almost surely.
But then, from \eqref{stp}, we must have
\begin{align*}
\sup_{0\leq r\leq   \tau_\infty   } 
E( u )(r)  \geq
\sup_{0\leq r\leq   \tau_k   } 
E( u )(r)  =k
\end{align*}

\noi
for any $k\geq 1$, almost surely.
This implies that  almost surely we have
\[
\sup_{0\leq r\leq   \tau_\infty   } 
E( u  ) (r)  =\infty.
\]

\noi
Thus, 
in view of \eqref{indsum}, we must have 
\begin{align*}
\ind_{ \{   \tau_\infty < T_0 -1 \}} =0
\end{align*}

\noi
almost surely.
Hence,  we can conclude 
\begin{align}
\mathbb{P} (   \{   \tau_\infty \geq  T_0 -1 \}  ) =1.
\label{goodP}
\end{align}

\noi
Since $T_0 $ is arbitrarily given, we may set $\wt T_0 =T_0-1 >0$ to be our new given time
such that \eqref{goodP} holds, thus, we finish the proof.

\end{proof}

\begin{lemma}
\label{LEM:st2}

Let $k\geq 1$ and $T_0>0$ be any deterministic value. Assume  $\tau_k$ is defined as in~\eqref{stp}.
We define a sequence of increasing  stopping times $  \g_k$ to be
\begin{align*}
\g_k= \inf \{  t\geq 0  :    \| \Psi\|_{X_1(t)} \geq k   \},
\end{align*}
where $X_1(t) $ is defined as in \eqref{X_1}. 
We also define the sets
\begin{align*}
\Si_k =\{  \o \,:\,   \tau_k \wedge \g_k   \geq  T_0  \}
\quad\quad
\text{and}
\quad\quad
\Si= \bigcup_{k \geq 1} \Si_k.
\end{align*}

\noi
Then,  we have $ \mathbb{P} ( \Si ) =1$.

\end{lemma}

\begin{proof}

It is equivalent to show that  $ \mathbb{P} ( \Si^c ) =0$.
Therefore, we compute 
\begin{align}
\begin{aligned}
\P( \Si^c) &= \P( \bigcap_{k\geq 1} \Si_k^c ) = \lim_{k \to \infty} \P( \Si_k^c)\\
&= \lim_{k \to \infty} \P(  \{  \o \,:\,   \tau_k \wedge \g_k   <  T_0  \}  ) \\
&= \lim_{k \to \infty} \P(  \{  \o \,:\,   \tau_k    <  T_0  \}  \cup   \{  \o \,:\,   \g_k   <  T_0  \}  ) \\
&\leq  \lim_{k \to \infty} \Big( \P(  \{  \o \,:\,   \tau_k    <  T_0  \} )
+  \P(  \{  \o \,:\,   \g_k   <  T_0  \}  ) \Big).
\end{aligned}
\label{bdset1}
\end{align}

\noi
By using Lemma \ref{LEM:st1}, we see from \eqref{bdset1} that
\begin{align}
\P( \Si^c)
\leq   \P( A^c )
+  \lim_{k \to \infty}  \P(  \{  \o \,:\,   \g_k   <  T_0  \}  ),
\label{bdse2}
\end{align}

\noi
where $\P( A^c ) =0$. Now, we show 
$\lim_{k \to \infty}  \P(  \{  \o \,:\,   \g_k   <  T_0  \}  ) =0$
in the following.

From the definition of $\g_k$, 
on the set $\{\o \,:\, \g_k < T_0  \}$,
we have
\begin{align*}
k\leq \| \Psi \|_{X_1(\g_k)} < \| \Psi \|_{X_1(T_0)},
\end{align*}

\noi
which implies 
\begin{align}
\{ \o \,:\, \g_k <T_0 \} \subset 
\{ \o \,:\,  \| \Psi \|_{X_1(T_0)}  >k \}.
\label{incl}
\end{align}

\noi
By using \eqref{incl}, Lemma \ref{regularity of Psi} (iii) (with $s=0$ and $\sigma=1$), 
and Markov's inequality,
 we obtain
\begin{align*}
\P(\{  \g_k <T_0 \} )  & \leq  \P (\{   \| \Psi \|_{X_1(T_0)}  >k \})\\
&\leq \frac1k \E[  \| \Psi \|_{X_1(T_0)} ] ,
\end{align*}

\noi
which, by taking the limit as $k\to \infty$ and using \eqref{bdse2},
allows us to conclude that $ \mathbb{P} ( \Si^c ) =0$. 
This finishes our proof.

\end{proof}

\begin{proof}[Proof of Theorem \ref{thm-d-i}]

In the following, we show that for any given (deterministic) time $T_0>0$
and for any given  initial data $(u_0,u_1)\in \dot\H^1$,  
solution of \eqref{eq_SHW} exists on $[0, T_0]$, almost surely.

\medskip
\noi
{\bf Step 1:}
Perturbative argument   for any given (deterministic) time $T_0>0$.

%
%

\medskip

For  $k$ to be fixed
and then for all $\o\in \Si_k$, 
given any   target time $T_0>0$,
by its
 definition, we have
\begin{align}
\sup_{0\leq r\leq T_0  }E( u  )(r)  \leq k
\quad
\text{and}
\quad
 \| \Psi \|_{X_1(T_0)} \leq k,
 \quad
 \text{on}
 \,\, \Si_k.
 \label{gdbdd}
 \end{align}

\noi
With \eqref{gdbdd} in hand, 
 we can run the following argument.

Let $u$ be a local-in-time solution to \eqref{eq_SHW} on the interval $[0,t_0]$
(which is guaranteed by the local well-posedness theory; 
see  Proposition~\ref{local} \& Remark~\ref{RM:local}). 
We aim to show that if $t_0<T_0$, there exists a
 $\tau>0$ 
 (independent of 
 $t_0$) such that 
$u$
 can be extended to $[0,t_0+\tau]$. 

We write $v=u-\Psi$, then $v$ is the solution to the following perturbed equation: 
\begin{equation}
\label{th1 eq1}
\begin{cases}
\pa_t^2 v-\Delta v+\mathcal{N}(v+\Psi)=0\\
(v,\pa_t v)|_{t=0}=(u_0,u_1),
\end{cases}
\end{equation}

\noi
where $\mathcal{N}(u)=(|\cdot|^{-4}\ast u^2)u$.
The main idea is to view  \eqref{th1 eq1} 
as a perturbation to the (deterministic) energy-critical HNLW
\eqref{HNLW3} on $\R^d$,
that is, regard $v$ as  in Lemma \ref{perturbation}
with perturbation
\begin{align}
\begin{aligned}
e&=
\mathcal{N}(v+\Psi)-\mathcal{N}(v)\\
&=( | \cdot |^{-4} \ast  (2v\Psi +\Psi^2) ) (v+\Psi) 
+( | \cdot |^{-4} \ast  |v|^2  ) (\Psi) .
\end{aligned}
\label{eq_pb}
\end{align}

\noi
Then, the  argument follows close to \cite{CLO21, OO20, BLL23}.

Let $w$ be the global solution to the (deterministic) energy-critical HNLW
\eqref{HNLW3}    
with the initial data $(w(t_0),\pa_t w(t_0))=(v(t_0),\pa_t v( t_0))$ 
(the global existence of solutions to \eqref{HNLW3}   is guaranteed in \cite{MZZ14}).
Then, from \eqref{gdbdd} and same initial data assumption
\begin{align*}
\| (w(t_0), \dt w(t_0))\|_{\dot\H^1} & \leq  E( w ) (t_0) 
= E( v )(t_0) \\
&\leq E(u-\Psi)(t_0)
\les    E(u)(t_0) + E(\Psi )(t_0)
\les k .
\end{align*}

\noi
Hence, from 
Lemma \ref{st-bound} (for the deterministic solutions),
 the following global space-time bound holds
\begin{equation*}
\|w\|_{X([0,T_0])}\leq\|w\|_{X(\R_+)}\leq C(k).
\end{equation*}

\noi
Moreover, Lemma \ref{regularity of Psi} gives that $\Psi \in X_1(T_0)$ almost surely. 
 In particular, for any $ I \subseteq [0, T_0] $, recalling that $ q > 3 $ in the definition of $ X_1$ and utilizing H\"older's inequality in time, we can obtain the following by virtue of  \eqref{gdbdd}:
\begin{equation}
\label{eq_stv1a}
\|\Psi\|_{X( I)} \leq 
|I|^\theta   
\|\Psi\|_{X_1(I)}  \leq 
|I|^\theta   
\|\Psi\|_{X_1(T_0)}  
\leq  
|I|^\theta k,
\end{equation}

\noi
almost surely, 
for some 
 $\theta>0 $.

We divide the interval $[t_0, T_0]$ 
into $J =J(k, \eta)$ 
many subintervals $I_j = [t_{j} , t_{j+1}]$ so that
\begin{equation}
\|w\|_{X(I_j)} 
\leq \eta.
\label{eq_wlo}
\end{equation}

\noi
Here, we note
that $0< \eta\ll1$ 
 will be chosen as an absolute constant and hence dependence of other
constants on $\eta$ is not essential in the following,
and also $I_j$ is deterministically chosen  
so that on each interval we have \eqref{eq_wlo}.
Fix $\tau>0$ to be chosen later.
We write 
\begin{align*}
[t_0, t_0+\tau]&=\bigcup_{j=0}^{J'}([t_0,t_0+\tau]\cap I_j  \cap [0,   \tau_k (\o) \wedge \g_k(\o) ] )\\
&=\bigcup_{j=1}^{J'}([t_1,t_1+\tau]\cap \wt I_j) 
\end{align*}

\noi
for some $J'\leq J$, 
where $[t_0,t_0+\tau]\cap I_j\neq\emptyset$ for $1\leq j\leq J'$.

Since the nonlinear evolution of $w$ on each $\wt I_j$ is small, 
therefore, 
the linear evolution $S(t-t_j)(w(t_j),\pa_t w(t_j))$
is also small on each $I_j$. 
 Indeed, the Duhamel formula implies
\begin{equation*}
S(t-t_j)(w(t_j),\pa_t w(t_j))=
w(t)-
\int_{t_j}^t
Q (t-t') (|\cdot|^{-4}\ast|w|^2)w(t') dt'.
\end{equation*}

\noi
Then, by the Strichartz estimate \eqref{Strichartz estimate},
Hardy-Littlewood-Sobolev inequality
 \eqref{HLS-f},
  and \eqref{eq_wlo}, 
we obtain
\begin{align}
\begin{aligned}
\|S(t-t_j)(w(t_j),\pa_t w(t_j))\|_{X( \wt I_j)}
 &=
 \Big\|w(t)+\int_{t_j}^t
 Q (t-t')
 (|\cdot|^{-4}\ast|w|^2)w(t') dt'\Big\|_{X(\wt  I_j)}\\
 &\leq
 \|w(t)\|_{X( \wt  I_j)}+C\|(|\cdot|^{-4}\ast|w|^2)w\|_{L_{\wt  I_j}^1L_x^2}\\
&\leq
\|w(t)\|_{X( \wt  I_j)}+C\|w\|_{X( \wt  I_j)}^3\leq 2\eta
\end{aligned}
\label{th1 est6}
\end{align}

\noi
for $j=0,1,\cdots,J'$, provided that $\eta>0$ is sufficiently small.

Now, we estimate $v$ on the first interval $  \wt  I_0=[t_0,t_1]  \cap [0,   \tau_k \wedge \g_k] $. 
We may assume that $\wt   I_0 \subseteq [t_0, t_0+\tau]$, if not,
we define a new time interval $[t_0, t_0+\tau] \cap \wt I_0 $.
By assumption
$(w(t_0),\pa_t w(t_0))=(v(t_0),\pa_t v(t_0))$ and 
\eqref{th1 est6}, 
we have
\begin{align*}
\begin{aligned}
\|S(t- t_0)(v (t_0),\pa_t v(t_0))\|_{X( \wt  I_0)}&=
\|S(t-t_0)(w(t_0),\pa_t w(t_0))\|_{X(\wt  I_0)}
\leq 2\eta.
\end{aligned}
\end{align*}

\noi
Then, by the local theory (Proposition \ref{local}), we have
\begin{equation*}
\|v\|_{X(\wt  I_0)}\leq 4\eta,
\end{equation*}

\noi
as long as  $\tau = \tau (\eta , \theta)  > 0$ is sufficiently small so that
\begin{equation*}
\|\Psi\|_{X( \wt  I_0)}  \leq 
|\wt  I_0|^\theta   
\|\Psi\|_{X_1(T_0)}  
\leq  
|\tau |^\theta  k \leq \eta.
\end{equation*}

Next, we estimate the error term. 
By using \eqref{eq_pb}, \eqref{Strichartz estimate}, and \eqref{eq_stv1a},
we repeat as in the proof of Proposition \ref{local} to obtain
\begin{equation}
\|e\|_{L_{\wt  I_0}^1L_x^2}
\leq  k (\eta + \tau^\theta )^{\al}
 \tau^\theta \leq k   \tau^\theta 
 \label{eq_er1}
\end{equation}

\noi
for some $\al=\al(d)$ depends only on the dimension,
and for any small $\eta , \tau > 0$.
Given $\eps > 0$, we can choose $\tau  = \tau  (\eps, \theta) > 0$ sufficiently small so that
\begin{equation*}
\|e\|_{L_{\wt  I_0}^1L_x^2}
\leq  \eps.
\end{equation*}

\noi
In particular, for $\eps < \wt \eps$ with $\wt \eps = \wt \eps(k) > 0$
dictated by Lemma \ref{perturbation}, the condition \eqref{p-condition2} is
satisfied on $I_0$. 
Hence, by the perturbation lemma (Lemma \ref{perturbation}), we obtain
\begin{equation*}
\|(v -w ,\pa_t v -\pa_t w )\|_{L^\infty_{ \wt  I_0} \dot\H^1}
+\|v-w\|_{L_{\wt  I_0}^q L_x^r} 
\leq C_1(k)\e  
\end{equation*}

\noi
for all $\dot H^1$-wave admissible pairs $(q,r)$.
In particular, we have
\begin{equation}
\|(v(t_1)-w(t_1),\pa_t v(t_1)-\pa_t w(t_1))\|_{\dot\H^1}
\leq C_1(k) \e.
\label{ent0}
\end{equation}

\noi
This means that at the right end point of $I_0$,
the energy remains finite, 
and it will be the initial energy for the next iteration step.

We now move onto the second interval $\wt  I_1=[t_1,t_2] \cap [0, \tau_k \wedge \g_k]$ and now $t_1$ is the initial time for the following computation.
 By \eqref{th1 est6} and \eqref{Strichartz estimate}  with \eqref{ent0}, we have
 \begin{align}
\begin{aligned}
\|S(t-t_1) (v (t_1),\pa_t v(t_1))\|_{X( \wt  I_1)}  &\leq 
\|S(t-t_1)  (w (t_1),\pa_t w(t_1))\|_{X(\wt  I_1)} \\
&\quad +
\|S(t-t_1) (w(t_1) - v(t_1), \pa_t w(t_1) -\pa_t v(t_1) )\|_{X(\wt  I_1)}\\
&\leq 2\eta + C_0 \cdot C_1(k) \e \leq 3\eta,
\end{aligned}
\label{li_2}
\end{align}

\noi
where $C_0$ is a pure constant from applying \eqref{Strichartz estimate}, and 
we have chosen  $\eps = \eps (k, \eta) =\eps(k) > 0$ sufficiently small.
Proceeding as before, it follows from local theory  (Proposition \ref{local})
 with \eqref{li_2} that
\begin{equation*}
\|v\|_{X(\wt  I_1)}\leq 6\eta,
\end{equation*}

\noi
as long as $\tau > 0$ is sufficiently small so that
\begin{equation*}
\|\Psi\|_{X(\wt   I_1)}  
 \leq 
|\wt  I_1|^\theta   
\|\Psi\|_{X_1(T_0)} 
\leq  
|\tau |^\theta k \leq \eta.
\end{equation*}

\noi
By repeating the
computation for the error term
in \eqref{eq_er1} with \eqref{eq_stv1a}, we have
\begin{equation*}
\|e\|_{L_{\wt  I_1}^1L_x^2}
\leq  k \tau^\theta\leq  \eps,
\end{equation*}

\noi
by choosing $\tau = \tau (\eps, \theta) > 0$ sufficiently small. 
Hence, by the  
perturbation lemma (Lemma~\ref{perturbation}) applied to the second interval $I_1$,
 we obtain
\begin{equation*}
\|(v -w ,\pa_t v -\pa_t w )\|_{L^\infty_{\wt  I_1} \dot\H^1}
+\|v-w\|_{L_{\wt  I_1}^q L_x^r} 
\leq C_1(k) ( C_1(k) + 1)  \e,
\end{equation*}

\noi
provided that $\tau  = \tau (\eps, \theta) > 0$ 
is chosen sufficiently small and that $(C_1(k) + 1)\eps < \wt \eps$. 
In particular, we have
\begin{equation*}
\|(v(t_2)-w(t_2),\pa_t v(t_2)-\pa_t w(t_2))\|_{\dot\H^1}
\leq C_1(k) ( C_1(k) + 1)  \e =: C_2(k) \eps.
\end{equation*}

For $j\geq 2$, 
define $C_j (k) $ recursively by setting
\begin{equation*}
C_j (k) = C_1(k) (C_{j-1}(k) +1 ).
\end{equation*}

\noi
Then, proceeding  above argument inductively, we obtain
\begin{equation*}
\|(v(t_j)-w(t_j),\pa_t v(t_j)-\pa_t w(t_j))\|_{\dot\H^1}
\leq  C_j(k) \eps
\end{equation*}

\noi
for all $0 \leq  j \leq  J'$, 
as long as $\eps = \eps (k, \eta , J) > 0$ 
is sufficiently small such that
\begin{itemize}
\item
$C_0 \cdot  C_j (k) \eps  \leq  \eta $
(here, $C_0$ is the constant from the Strichartz estimate in \eqref{li_2}),

\item
$(C_j (k) + 1) \eps < \eps_0$ so that 
perturbation lemma (Lemma \ref{perturbation}) applies,

\end{itemize}

\noi
for $j = 1, \cdots,J'$. 
Recalling that $J' + 1 \leq J = J(k, \eta)$, 
we see that this can be achieved by
choosing small $\eta > 0$, $\eps = \eps(R, \eta) = \eps(k) > 0$,
 and $\tau = \tau (\eps, \theta) = \tau (k, \theta) > 0$ sufficiently small. 
 This guarantees existence of a (unique) solution $v$ to  \eqref{th1 eq1} on 
 $[t_0, t_0 + \tau]$. 
 Lastly,
noting that $\tau > 0$ is independent of $t_0 \in [ 0, T_0 )$, 
we conclude existence of the solution $v$
to  \eqref{th1 eq1}  on the entire closed  interval $ [0,   \tau_k (\o) \wedge \g_k (\o)]
\cap [0, T_0 ]$ for all $\o \in \Si_k$.

Note the above argument works for any $k$.
Thus, for any given $T_0>0$, we have
\begin{align}
\bigcup_{k} \Si_k \subset \{ v \,\, \text{exists till}\,\, T_0  \}.
\label{eq_clu}
\end{align}

\noi
By Lemma \ref{LEM:st2} and \eqref{eq_clu}, we obtain
\[
\P ( \{ v \,\, \text{exists till}\,\, T_0  \})=1.
\]

\medskip
\noi
{\bf Step 2:}
Conclusion.

\medskip
Then, in view of the perturbation argument in step 1, 
if $T^*(\o, u_0, u_1, \phi)$ happens before $\tau_k \wedge \g_k$, then 
we can always extend solution up to $T^*$ so that 
\[
\| v \|_{X(T^*)} \leq C(k,\eta)<\infty, 
\]

\noi
which is a contradiction with the assumption that 
$T^*$ is the maximal existence time and Proposition \ref{blow-up}.
This proves that $T^* \ge \tau_k \wedge \g_k$, almost surely.

We  set 
\[
\Si_{T_0}=\{  \tau_\infty \wedge \g_\infty  \geq T_0   \} ,
\]

 \noi
 where $\g_\infty = \lim_{k\to \infty} \g_k$.
In Lemma \ref{LEM:st2}, 
 we have shown for any $T_0>0$, $\P(\Si)=1$.
 Moreover, 
 $\Si \subset \Si_{T_0}$, thus,
$ \P(\Si_{T_0}) =1$.
 Now, we relabel the target time to be $T_j=2^j$ and 
 we have $\P(\Si_{T_j}) =1$ for $j \geq 1$.
 Let  $\Si_{\infty} = \bigcap_{j\geq 1} \Si_{T_j}$,
 then $\mathbb P(\Si_{\infty}) =1$.

\end{proof}

\begin{remark}
\rm

In order to establish global-in-time existence, 
we show that the solution exists up to any given  
{\it deterministic} target time $T_0 > 0$.
 However, 
 we cannot work with such a fixed target time directly. 
 Instead, it is necessary to introduce a stopping time adapted to the {\it filtration generated by the noise} (since the initial data is deterministic), 
 and work on the intersection of the given target time and the stopping time.
%
%

\end{remark}

\subsection{Energy bound with random initial data}

In this subsection, we show the uniform probabilistic a priori energy bound on smooth approximating solutions.

Given $(u_0,u_1)\in\H^s(\R^5)$ with $\frac{1}{2}<s<1$.
Let $N\geq 1$ dyadic,
we define $u_{j,N}^\omega$ by
\begin{equation*}
u_{j,N}^\omega:= P_{\leq N}u_j^\omega
=
\sum_{n\in\mathbb{Z}^5}g_{n,j}(\omega)P_{\leq N}\psi(D-n)u_j
\quad
\text{for }
 j=0,1.
\end{equation*}

\noi
It is clear that  $(u_{0,N}^\omega, u_{1,N}^\omega)\in \H^\infty(\R^5)$.
Assume 
$\phi\in\HS(L^2;H^1)$.
Let $u_N=u_N^\o$ be the smooth global-in-time solution to
 \eqref{eq_SHW}
with (random) initial data 
$
 (u_N,\pa_t u_N)|_{t=0}=(u_{0,N}^\omega,u_{1,N}^\omega)
$ 
and $\phi$ replaced by $P_{\leq N}\phi$.
We  
follow the discussion as in the introduction,
decompose $u_N$ into
\begin{equation}
\label{z_N}
u_N=v_N+\wt \Psi_N,
\end{equation}

\noi
where $\wt \Psi_N=S(t)(u_{0,N}^\om,u_{1,N}^\om) +\Psi_N$, and $v_N=v_N^\o$ is the residual term.
Here, $\Psi_N$ is defined by \eqref{Psi'} with $\phi$ replaced by $P_{\leq N}\phi$.
The probabilistic Strichartz estimates (Lemma \ref{P-est-S*}) 
and the regularity properties of the convolution (Lemma \ref{regularity of Psi}) 
ensure that the above decomposition \eqref{z_N} holds true.
More precisely, we have \eqref{regpsi2}.
The residual term and $v_N$ is the smooth global solution to \eqref{eq_SHW3}
with
$(v_N,\pa_tv_N)|_{t=0}=(0,0)$.
For the convenience,
we also denote the (random) linear solution 
\begin{align}
\label{def t-z_N}
z_N=z_N^\o=S(t)(u_{0,N}^\om,u_{1,N}^\om)
\quad
\text{and}
\quad
\wt z_N=\wt z_N^\omega:=\wt S(t)(u_{0,N}^\omega,u_{1,N}^\omega),
\end{align}

\noi
where $\wt S(t)$ is defined in \eqref{t-S(t)}. 
In view of \eqref{S(t) t-S(t)},
we have $\pa_t z_N=\langle\nabla\rangle\wt z_N$. 

In terms of the above conventions, we establish the following  energy bound of $E(v_N)$ 
for solutions $v_N$ of equation \eqref{eq_SHW3}, 
on the condition that the linear solution $z_N$ meets certain conditions.

\begin{proposition}
\label{D-bound}
Let $ \frac{1}{2}<s<1$ and $N\geq 1$ dyadic. 
Given $T>0$ and assume that exists a constant $K>0$ such that
\begin{align}
\label{zN t-zN bound}
\|\wt\Psi_N\|_{Z(T)}<K,\quad\|\wt z_N\|_{Y(T)}<K,
\quad
\text{and}
\quad\|\pa_t\Psi_N\|_{L_t^{10}L_x^{\frac{10}{3}}([0,T]\times\R^5)}<K,
\end{align}

\noi
where 
$Y(T)$ and $Z(T)$ are defined in
\eqref{def_WY}.
Then, there exists a positive constant $C(K,T)$ 
which is dependent on both $K$ and $T$
such that
\begin{equation}
\label{goal1}
E(v_N(t))\leq C(K,T).
\end{equation}

\noi
In particular, we have
\begin{equation}
\label{goal2}
\sup_{t\in[0,T]}\|(v_N(t),\pa_tv_N(t))\|_{\H^1(\R^5)}^2 \leq C(K,T)
\end{equation}
for the solution $v_N$ to \eqref{eq_SHW3}.
\end{proposition}

\begin{proof}
Taking the time derivative of the energy, using \eqref{eq_SHW3}
and integrating with respect to time, 
we obtain 
\begin{equation}
\label{E=I+J}
\begin{split}
E(v_N(t))
&=-\int_0^t\int_{\R^5}\pa_{t'} v_N[(|\cdot|^{-4}\ast |\wt\Psi_N|^2)(v_N+\wt\Psi_N)+(|\cdot|^{-4}\ast (2v_N\wt\Psi_N))\wt\Psi_N]dxdt' \\
&\quad -\int_0^t\int_{\R^5}\pa_{t'} v_N[(|\cdot|^{-4}\ast |v_N|^2)\wt\Psi_N+(|\cdot|^{-4}\ast (2v_N\wt\Psi_N))v_N]dxdt'\\
&=: \1(t)+J(t). 
\end{split}
\end{equation}

\noi
In what follows,
we will omit the subscript of $N$ from our notation  for simplicity.

First, we establish a key estimate based on the Hartree potential energy. 
From interpolation and Lemma \ref{Riesz-Cha}, we have
\begin{align}
\label{10/3-norm}
\begin{aligned}
\|v\|_{L_x^{\frac{10}{3}}}
\lesssim
\||\nabla|v\|_{L_x^2}^{\frac{1}{5}}\||\nabla|^{-\frac{1}{4}}v\|_{L_x^4}^{\frac{4}{5}}
\lesssim 
\||\nabla|v\|_{L_x^2}^{\frac{1}{5}}\||\nabla|^{-\frac{1}{2}}|v|^2\|_{L_x^2}^{\frac{2}{5}}
\lesssim E(v(t))^{\frac{3}{10}}.
\end{aligned}
\end{align}

\noi
For $\1(t)$, we apply H\"{o}lder's, 
 Hardy-Littlewood-Sobolev inequalities, and \eqref{10/3-norm}, to deduce
\begin{equation}
\label{I}
\begin{split}
|\1(t)|
%
&\lesssim
\int_0^t \big(E(v(t'))+\|\wt\Psi\|^6_{L_x^{\frac{30}{7}}}+E(v(t'))\|\wt\Psi\|_{L_x^5}^2\big) d t'.
\end{split}
\end{equation}
For $J(t)$, we need to employ the integration by parts trick of \cite{OP16}:
\begin{equation*}
\begin{split}
J(t)
= -\int_{\R^5}v^2(|\cdot|^{-4}\ast (v\wt\Psi)) dx+\int_{\R^5}\int_0^t v^2(|\cdot|^{-4}\ast (v\pa_{t'} \wt\Psi)) dt' dx
=: J_1(t)+J_2(t).
\end{split}
\end{equation*}

\noi
Now, we consider $J_1(t)$.
\noi
By fixing a constant $0<\e_0\ll 1$ sufficiently small and to be chosen later,
using \eqref{10/3-norm},
Hardy-Littlewood-Sobolev \eqref{HLS-f},
 H\"{o}lder's,
and Young's inequalities,
we can obtain 
\begin{equation}
\label{J1}
\begin{split}
|J_1(t)|
\lesssim
\|v\|_{L_x^{\frac{10}{3}}}^3\|\wt\Psi\|_{L_x^{\frac{10}{3}}}
\lesssim E(v(t))^{\frac{9}{10}}\|\wt\Psi\|_{L_x^{\frac{10}{3}}}
\lesssim \frac{9}{10}\e_0 E(v(t))+\frac{1}{10}\e_0^{-9}\|\wt\Psi\|_{L_x^{\frac{10}{3}}}^{10}.
\end{split}   
\end{equation}

\noi
For the second term $J_2(t)$, 
from \eqref{def t-z_N}, we have 
\begin{equation}
\label{J_2}
\begin{split}
J_2(t)
&=\int_{\R^5}\int_0^t v^2(|\cdot|^{-4}\ast (v\langle\nabla\rangle\wt z)) dt' dx
+\int_{\R^5}\int_0^t v^2(|\cdot|^{-4}\ast (v\pa_{t'} \Psi)) dt' dx\\
&=: J_{2}^{(1)}(t)+J_2^{(2)}(t).
\end{split}
\end{equation}

\noi
For the first term $J_{2}^{(1)}(t)$, we define $\cI(t)$ by
\begin{align*}
\mathcal{I}(t)
=\int_{\R^5}\langle\nabla\rangle\wt zv(|\cdot|^{-4}\ast v^2) dx
=\sum_{k=-1}^1\sum_{\substack{M\geq 1\,  \text{dyadic}}}M\int_{\R^5}P_{2^kM}\wt z \, P_M[v(|\cdot|^{-4}\ast v^2)] dx   
\end{align*}
with the understanding that $P_{2^{-1}}=0$ and 
$J_2^{(1)}(t)=\int_0^t\mathcal{I}(t')  d t'$.

{\bf Case 1}: $M=1$. 
Similar to the proof of \eqref{J1},  we have
\begin{equation}
\label{J_21 M=1}
\begin{split}
|J_2^{(1)}(t)|
&=
\Big|\int_0^t\int_{\R^5}
\big(P_1\wt z(t')P_1[v(|\cdot|^{-4}\ast v^2)]+P_2\wt z(t')P_1[v(|\cdot|^{-4}\ast v^2)] \big) dx dt'\Big|\\
&\lesssim
\int_0^t\frac{9}{10}\e_0 E(v(t')) dt'+\int_0^t\frac{1}{10}\e_0^{-9}\|\wt z\|_{L_x^{\frac{10}{3}}(\R^5)}^{10} dt'.
\end{split}
\end{equation}

{\bf Case 2}: 
$M\geq 2$. Using the (inhomogeneous) Littlewood-Paley again, we have 
\begin{equation*}
v(|\cdot|^{-4}\ast v^2)=\sum_{\substack{M_j\geq 1\,   \text{dyadic}}}P_{M_1}v(|\cdot|^{-4}\ast (P_{M_2}vP_{M_3}v)).
\end{equation*}

\noi
The proofs of $M_1\geq M_2\geq M_3$ 
and $M_1\leq M_2\leq M_3$ are similar, 
so we only consider $M_1\geq M_2\geq M_3$.
Note that $P_M[v(|\cdot|^{-4}\ast v^2)]=0$ unless $M_1\gtrsim M$. With $M_1\gtrsim M$, we have
\begin{equation*}
\mathcal{I}(t)\sim\sum_{k=-1}^1\sum_{M\geq 2}\sum_{\substack{M_1,M_2,M_3\\   M_1\gtrsim M}} M\int_{\R^5}P_{2^kM}\wt z(t)P_M[P_{M_1}v(|\cdot|^{-4}\ast (P_{M_2}vP_{M_3}v))]dx.
\end{equation*}
Using H\"{o}lder's inequality, we derive
\begin{equation*}
\begin{split}
|\mathcal{I}(t)|
&\lesssim
\sum_{k=-1}^1\sum_{M\geq 2}\sum_{\substack{M_1,M_2,M_3\\   M_1\gtrsim M}}
\big\|\langle\nabla\rangle^{s-\delta}P_{2^kM}\wt z(t)\big\|_{L_x^5}
M_1^{1-s+\delta} \big\|P_{M_1}v(|\cdot|^{-4}\ast (P_{M_2}vP_{M_3}v))\big\|_{L_x^{\frac{5}{4}}} \\
&\lesssim
\sup_{M_1,M_2,M_3}
\big\| \langle\nabla\rangle^{s-\delta}\wt z(t)\big\|_{L_x^5}
M_1^{1-s+\delta+}
\big\|P_{M_1}v(|\cdot|^{-4}\ast (P_{M_2}vP_{M_3}v))\big\|_{L_x^{\frac{5}{4}}}.
\end{split}
\end{equation*}
By H\"{o}lder's,
Hardy-Littlewood-Sobolev inequalities, interpolation and \eqref{10/3-norm},
we see
\begin{equation*}
\begin{split}
M_1^{1-s+\delta+} \|P_{M_1}v(|\cdot|^{-4}\ast (P_{M_2}vP_{M_3}v))\|_{L_x^{\frac{5}{4}}}
&\lesssim
M_1^{1-s+\delta+}(\|P_{M_1}v\|^{\frac{1}{2}}_{L_x^2}\|P_{M_1}v\|^{\frac{1}{2}}_{L_x^{\frac{10}{3}}})\|P_{M_2}vP_{M_3}v\|_{L_x^{\frac{5}{3}}}\\
&\lesssim
\|M_1^{2(1-s+\delta+)}P_{M_1}v\|^{\frac{1}{2}}_{L_x^2}\|v\|^{\frac{5}{2}}_{L_x^{\frac{10}{3}}}\\
&\lesssim
E(v(t))^{\frac{1}{4}}E(v(t))^{\frac{3}{4}},
\end{split}
\end{equation*}

\noi
where the last inequality follows from Bernstein's inequality as long as $2(1-s+\delta+)\leq 1$, i.e. $s>\frac{1}{2}+\delta.$ Hence, we have
\begin{align}
\label{J21 M>2 1}
\begin{split}
|\mathcal{I}(t)|\les  
\|\langle\nabla\rangle^{s-\delta}\wt z(t)\|_{L_x^5}E(v(t)).
\end{split}
\end{align}


\noi
Therefore, 
we conclude $J^{(1)}_2(t)$ by
using \eqref{J_21 M=1}
and \eqref{J21 M>2 1},
such that
\begin{equation}
\label{J_21}
\begin{split}
|J^{(1)}_2(t)|
&\lesssim
\int_0^t\frac{9}{10}\e_0 E(v(t')) dt'+\int_0^t\frac{1}{10}\e_0^{-9}\|\wt z\|_{L_x^{\frac{10}{3}}}^{10} dt'
+\int_0^t\|\langle\nabla\rangle^{s-\delta}\wt z(t')\|_{L_x^5}E(v(t')) dt'.
\end{split}
\end{equation}

\noi
As for $J_2^{(2)}(t)$ in \eqref{J_2}, 
by applying a similar  argument of
\eqref{J_21 M=1}, we obtain
\begin{equation}
\label{J_22}
|J_2^{(2)}(t)|
\lesssim 
\int_0^t\frac{9}{10}\e_0 E(v(t'))dt'
+\int_0^t
\frac{1}{10}\e_0^{-9}\|\pa_{t'}\Psi\|_{L_x^{\frac{10}{3}}}^{10}dt'.
\end{equation}

\noi
Combining the estimates \eqref{E=I+J},  \eqref{I}, \eqref{J1}, \eqref{J_2}, \eqref{J_21}, and \eqref{J_22}  together, 
choosing $\e_0$ sufficiently small such that $\frac{9}{10}\e_0 E(v(t))\leq\frac{1}{2}E(v(t))$,
we yield
\begin{equation*}
\begin{split}
E(v(t))
\les K+(1+ K)\int_0^t E(v(t'))  dt'.
\end{split}
\end{equation*}

\noi
Therefore,
 Gronwall's inequality implies 
  the energy bound \eqref{goal1} for all solutions $v_N$ to \eqref{eq_SHW3}.
As for \eqref{goal2}, it follows by, for $t\in[0,T]$
\begin{equation*}
\|(v_N(t),\partial_t v_N(t))\|_{\dot \H^1}
^2
\leq C(K,T),
\quad
\|v_N(t)\|_{L_x^2}\leq C(K,T).
\end{equation*}
This completes the proof of Proposition \ref{D-bound}.
\end{proof}

Proposition \ref{D-bound} is the main difficulty in the proof of global-in-time existence.
Subsequently, we give some remarks in the following to contrast the disparities between NLW and HNLW.
Additionally, we also identify the reason of dimension restriction on $d=5$.

\begin{remark}
\label{R-NLW-vN}
\rm 

We consider
the perturbed  energy-critical NLW
\begin{equation*}
\begin{cases}
\pa_t^2 v_N-\Delta v_N+|v_N+z_N|^{\frac{4}{3}}(v_N+z_N)=0\\
(v_N,\pa_tv_N)|_{t=0}=(0,0),
\end{cases}
\quad (t,x)\in\R\times\R^5,
\end{equation*}

\noi
which is associated with 
the energy $E_{{\rm NLW}}(v_N)$ by
\begin{equation*}
E_{{\rm NLW}}(v_N):=\frac{1}{2}\int_{\R^5}(|\nabla v_N|^2+|\pa_t v_N|^2)dx
+\frac{3}{10}\int_{\R^5}|v_N|^{\frac{10}{3}}dx.
\end{equation*}

\noi
By H\"{o}lder's inequality, 
the fact that 
\begin{align}
\label{ENGctr}
\|v_N\|_{L_x^{\frac{10}{3}}}^{\frac{4}{3}}
\les E_{{\rm NLW}}(v_N)^{\frac{2}{5}},
\end{align}

\noi
and Gronwall inequality,
it is not difficult to show that 
\begin{equation*}
\sup_{t\in[0,T]}\|(v_N(t),\pa_tv_N(t))\|_{\H^1}\leq E_{{\rm NLW}}(v_N)\leq C.
\end{equation*}

\noi
Therefore,
it is sufficient to assume that the perturbed term satisfies 
\begin{equation*}
z_N\in L_t^{\frac{7}{3}}L_x^{\frac{14}{3}}([0,T]\times\R^5)\cap L_t^1L_x^{10}([0,T]\times\R^5),
\end{equation*}

\noi
then
 the Gronwall-type argument developed by Burq-Tzvetkov \cite[Proposition 2.2]{BT14}
 is  enough to establish the energy bound
 to energy-critical NLW  on $\R^5$.
 
%
%

%
%
\end{remark}

\begin{remark}\label{difficult}
\rm

In the case of Hartree nonlinearity,
the energy alone is insufficient to complete the proof of Proposition \ref{D-bound}. 
Namely, a direct application of
the Gronwall-type argument of \cite[Proposition 2.2]{BT14} is not enough to 
establish the energy bound. 
The potential energy of $E(v_N(t))$ 
cannot be used to control $\|v_N\|_{L_x^{10/3}}$, 
as we saw in NLW case of \eqref{ENGctr},
 where
\begin{equation*}
E(v_N(t))
:= 
\frac{1}{2}\int_{\R^5}(|\nabla v_N(x)|^2+|\pa_t v_N(x)|^2)dx+\frac{1}{4}\int_{\R^5}\int_{\R^5}\frac{|v_N(x)|^2|v_N(y)|^2}{|x-y|^4}dxdy.
\end{equation*}

\noi
This issue  is similar to the energy-critical NLW on $\R^3$;
see \cite[Remark 5.1]{OP17} and 
it was resolved in \cite{OP16}.
Therefore,
the above proof follows the idea of 
\cite[Proposition 4.1]{OP16}.
We show that 
by assuming some appropriate regularity properties of the perturbed terms,
we can control the energy $E(v_N(t))$, 
uniformly in $N$.
Later on we show the assumption 
\eqref{zN t-zN bound} holds true, almost surely;
see Proposition \ref{R-W,Y}.

\end{remark}

\begin{remark}\label{d=5}
\rm

For $d>5$, the above computation fails to establish the upper bound of $E(v_N,\dt v_N)$. 
%
To see this, we consider the estimate for the term $\mathcal{I}(t)$ in {\bf Case 2}. 
First, using~\eqref{Riesz},  \eqref{visan lem} and the interpolation, we have
\begin{equation}
\label{v E(v)-d}
\|v\|_{L_x^{\frac{2d}{d-2}}}\lesssim E(v(t))^{\frac{d-2}{2d}}.
\end{equation}
By H\"{o}lder's inequality, there exist $1<m_1,m_2<\infty$ satisfy $\frac{1}{m_1}+\frac{1}{m_2}=1$ such that
\begin{equation*}
\begin{split}
|\mathcal{I}(t)|\lesssim
\|\langle\nabla\rangle^{s-\delta}\wt z(t)\|_{L_x^{m_1}}
M_1^{1-s+\delta+}\|P_{M_1}v(|\cdot|^{-4}\ast (P_{M_2}vP_{M_3}v))\|_{L_x^{m_2}}.
\end{split}
\end{equation*}

\noi
Then, by H\"{o}lder's,
Hardy-Littlewood-Sobolev inequalities,
\eqref{v E(v)-d} 
and choosing $0\leq\theta<1$ satisfies 
$\theta<s-\delta$,
we obtain
\begin{equation*}
\|M_1^{\frac{1-s+\delta+}{1-\theta}}P_{M_1}v\|_{L_x^2}\leq E(v(t))^{\frac{1}{2}}, 
\end{equation*}
where $m_2=\frac{2d}{3d-2\theta-4}$.
Hence, we would have
\begin{equation*}
\begin{split}
M_1^{1-s+\delta+}\|P_{M_1}v(|\cdot|^{-4}\ast (P_{M_2}vP_{M_3}v))\|_{L_x^{m_2}}  
&\leq
E(v(t))^{\frac{1-\theta}{2}}E(v(t))^{\frac{(2+\theta)(d-2)}{2d}}.
\end{split}
\end{equation*}

\noi
However, to apply 
Gronwall's inequality, the condition
 $\frac{1-\theta}{2}+\frac{(2+\theta)(d-2)}{2d}\leq 1$ must be met. 
By solving this inequality, we find that it implies $d\leq 5$. 
Moreover, for $m_1>1$, it also implies~$d\leq 5$.
\end{remark}

Next, we show that
conditions of Proposition \ref{D-bound} hold,
almost surely. 
%
Using Lemma \ref{regularity of Psi}, Lemma \ref{P-est-S*} and the Bernstein inequality, we have the following proposition.

\begin{proposition}\label{R-W,Y}
Let $\frac{1}{2}<s<1$, $z_N$ and $\wt z_N$ are defined in  \eqref{def t-z_N}. 
Then, for fixed $T,\e>0$, there exists $\lambda=\lambda(T,\e,\|(u_0,u_1)\|_{\H^s})>0$ such that
\begin{equation*}
\mathbb{P}(\{\omega\in\Omega:\|\wt\Psi_N\|_{Z(T)}+\|\wt z_N\|_{Y(T)}+\|\pa_t\Psi_N\|_{L_T^{10}L_x^{\frac{10}{3}}}> \lambda\})<\e.
\end{equation*}
\end{proposition}

At this stage, we are able to establish an 
energy bound for the solution to \eqref{eq_SHW3} at the truncation level
by combining the results of
 Proposition \ref{D-bound} and Proposition \ref{R-W,Y}.
In particular, we only prove a probabilistic
energy bound, uniformly in $N$, for the approximating random solutions $v^\o_N$.

\begin{proposition}
\label{energy-b}
Let $1/2<s<1$ and $N\geq 1$ dyadic. 
Given $T,\e>0$, there exists 
$\wt\Omega_{N,T,\e}\subset\Omega$ such that
\smallskip
\begin{enumerate}[{\rm (i)}]
\item $P(\wt\Omega_{N,T,\e}^c)<\e.$

\medskip
\item There exists a constant $C=C(T,\e,\|(u_0,u_1)\|_{\H^s(\R^5)})>0$, which is finite,
such that 
for any solution $v_N^\omega$ to \eqref{eq_SHW3} with $\om\in\wt\Omega_{N,T,\e}$,
the following energy bound holds:
\begin{equation*}
\sup_{t\in[0,T]}\|(v_N^\omega(t),\pa_tv_N^\omega(t))\|_{\H^1(\R^5)}\leq C.
\end{equation*}

\end{enumerate}
\end{proposition}

\begin{proof}

We define 
 $\wt\Omega_{N,T,\e}\subset\Omega$ by
\begin{equation*}
\wt\Omega_{N,T,\e}=
\{\omega\in\Omega:\|\wt\Psi_N\|_{Z(T)}+\|\wt z_N^\omega\|_{Y(T)}+\|\pa_t\Psi_N\|_{L_T^{10}L_x^{\frac{10}{3}}}\leq \wt\lambda\},
\end{equation*}

\noindent
where $\wt\lambda$ is given as in Proposition \ref{R-W,Y}.
Then, the proof follows from Proposition \ref{R-W,Y} and Proposition \ref{D-bound}.

\end{proof}

\subsection{Proof of Theorem \ref{thm-r-i}}

From Remark \ref{RM:local}, it remains to prove part (ii) of Theorem \ref{thm-r-i}. 
The proof follows directly from the argument used in \cite[Proposition 6.1]{OP16}, 
adapted to our context using Propositions \ref{condition LWP} and \ref{energy-b}. 
The details are omitted here for brevity.

\begin{ackno}\rm

The authors extend their sincere gratitude to Tadahiro Oh 
for his support and advice throughout the entire project, essentially for his invaluable suggestion regarding the consideration of stochastic noises.
G.L.~was supported by the EPSRC New Investigator Award (grant no. EP/S033157/1)
and the European Research Council (grant no. 864138 ``SingStochDispDyn'').
T.Z.~was supported by the National Natural Science Foundation of China (Grant No. 12101040, 12271051, and 12371239) and by a grant from the China Scholarship Council (CSC). T.Z.~would also like to thank the 
University of Edinburgh for its hospitality where this manuscript was prepared. 
The authors would like to thank the
anonymous referee for their helpful comments which have greatly improved the presentation
of this manuscript.
\end{ackno}





\begin{thebibliography}{99}
\bibitem{BCD11}
H.~Bahouri, J.~Chemin, R.~Danchin,
 \emph{Fourier analysis and nonlinear partial differential equations}, Grundlehren der mathematischen Wissenschaften [Fundamental Principles of Mathematical Sciences]
Springer, Heidelberg, 2011. xvi+523 pp.


\bibitem{BG99}
H.~Bahouri, P.~G\'erard,
\emph{High frequency approximation of solutions to critical nonlinear wave equations}, Amer. J. Math. 121 (1999), no. 1, 131--175.


\bibitem{BOP15}
\'A.~B\'enyi, T.~Oh, O.~Pocovnicu, \emph{Wiener randomization on unbounded domains and an application to almost sure well-posedness of NLS}, Excursions in harmonic analysis. Vol. 4, 3--25, Appl. Numer. Harmon. Anal., Birkh\"{a}user/Springer, Cham, 2015.



\bibitem{BOP15'}
\'A.~B\'enyi, T.~Oh, O.~Pocovnicu, \emph{On the probabilistic Cauchy theory of the cubic nonlinear Schr\"{o}dinger equation on $\mathbb{R}^d$, $d\geq3$}, Trans. Amer. Math. Soc. Ser. B 2 (2015), 1--50.

\bibitem{BOP-19}
\'A.~B\'enyi, T.~Oh, O.~Pocovnicu, 
\emph{Higher order expansions for the probabilistic local Cauchy theory of the cubic nonlinear Schr\"odinger equation on $\R^3$}, Trans. Amer. Math. Soc. Ser. B 6 (2019), 114--160.

\bibitem{BOP-19-2}
\'A.~B\'enyi, T.~Oh, O.~Pocovnicu,
\emph{On the probabilistic Cauchy theory for nonlinear dispersive PDEs}, Landscapes of time-frequency analysis, 
 1--32. Appl. Numer. Harmon. Anal.
Birkh\"auser/Springer, Cham, 2019






\bibitem{B94}
J.~Bourgain, \emph{Periodic nonlinear Schr\"{o}dinger equation and invariant measures}, Comm. Math. Phys. 166 (1994), no. 1, 1--26.

\bibitem{B96}
J.~Bourgain, \emph{Invariant measures for the $2$D-defocusing nonlinear Schr\"{o}dinger equation}, Comm. Math. Phys. 176 (1996), no. 2, 421--445.

\bibitem{B97}
J.~Bourgain, \emph{Invariant measures for the Gross-Piatevskii equation}, J. Math. Pures Appl. (9) 76 (1997), no. 8, 649--702.



\bibitem{Bringmann21}
B.~Bringmann, 
\emph{
Almost sure local well-posedness for a derivative nonlinear wave equation},
Int. Math. Res. Not. IMRN 2021, no. 11, 8657--8697.

\bibitem{BK00}
P.~Brenner, P.~Kumlin, \emph{On wave equations with supercritical nonlinearities}, Arch. Math. (Basel) 74 (2000), no. 2, 129--147.

\bibitem{BLL23}
E.~Brun, G.~Li, R.~Liu, \emph{Global well-posedness of the energy-critical stochastic nonlinear wave equations}, J. Differential Equations 397 (2024), 316--348.


\bibitem{BT07}
N.~Burq, N.~Tzvetkov, \emph{Invariant measure for a three dimensional nonlinear wave equation}, Int. Math. Res. Not. IMRN (2007), no.22, Art. ID rnm108, 26 pp.


\bibitem{BT08-1}
N.~Burq, N.~Tzvetkov,
\emph{Random data Cauchy theory for supercritical wave equations}, I. Local theory. Invent. Math. 173 (2008), no. 3, 449--475.


\bibitem{BT08-2}
N.~Burq, N.~Tzvetkov,
\emph{Random data Cauchy theory for supercritical wave equations}, II. A global existence result. Invent. Math. 173 (2008), no. 3, 477--496.

\bibitem{BT14}
N.~Burq, N.~Tzvetkov,
\emph{Probabilistic well-posedness for the cubic wave equation}, J. Eur. Math. Soc. (JEMS) 16 (2014), no. 1, 1--30. 


\bibitem{CCMNS20}
S.~Chanillo, M.~Czubak, D.~Mendelson, A.~Nahmod, G.~Staffilani,
\emph{Almost sure boundedness of iterates for derivative nonlinear wave equations},
Comm. Anal. Geom. 28 (2020), no. 4, 943--977.


\bibitem{CG13}
X.~Cheng, Y.~Gao,
\emph{Small data global well-posedness for the nonlinear wave equation with nonlocal nonlinearity}, Math. Methods Appl. Sci. 36 (2013), no. 1, 99--112.

\bibitem{CL22}
K.~Cheung, G.~Li,
\emph{Global well-posedness of the 4-D energy-critical stochastic nonlinear Schr\"odinger equations with non-vanishing boundary condition}, 
Funkcial. Ekvac. 
65 (2022), no. 3, 287--309.


\bibitem{CLO21}
K.~Cheung, G.~Li, T. Oh,
{\it Almost conservation laws for stochastic nonlinear Schr\"odinger equations}, 
J. Evol. Equ. 21 (2021), no. 2, 1865--1894.




\bibitem{CCT03}
M.~Christ, J.~Colliander, T.~Tao, \emph{Ill-posedness for nonlinear Schr\"{o}dinger and wave equations},
arXiv:math/0311048 [math.AP].


\bibitem{CKSTT08}
 J.~Colliander, M.~Keel, G.~Staffilani, H.~Takaoka, T.~Tao, \emph{Global well-posedness
and scattering for the energy-critical nonlinear Schr\"{o}dinger equation in $\R^3$}, Ann. of Math. (2) 167 (2008), no. 3, 767--865. 

\bibitem{CO12}
J.~Colliander, T.~Oh,
\emph{Almost sure well-posedness of the cubic nonlinear Schr\"{o}dinger equation below
$L^2(\mathbb{T})$}, Duke Math. J. 161 (2012), no. 3, 367--414.


\bibitem{DD02}
 G.~Da Prato, A.~Debussche, 
 {\it Two-dimensional Navier-Stokes equations driven by a space-time white noise},
  J. Funct. Anal. 196 (2002), no. 1, 180--210.




\bibitem{DZ14}
G.~Da Prato, J.~Zabczyk,
\emph{Stochastic equations in infinite dimensions.}
Second edition. Encyclopedia Math. Appl., 152
Cambridge University Press, Cambridge, 2014. xviii+493 pp.


\bibitem{D12}
Y.~Deng,
\emph{Two-dimensional nonlinear Schr\"{o}dinger equation with random radial data}, Anal. PDE 5 (2012), no. 5, 913--960. 

\bibitem{DLM20}
B.~Dodson, J.~L\"{u}hrmann, D.~Mendelson, \emph{Almost sure scattering for the 4D energy-critical defocusing nonlinear wave equation with radial data}, Amer. J. Math. 142 (2020), no. 2, 475--504.


\bibitem{FO20}
J.~Forlano, M.~Okamoto, \emph{
A remark on norm inflation for nonlinear wave equations},
Dyn. Partial Differ. Equ. 17 (2020), no. 4, 361--381.



\bibitem{G90}
M.~Grillakis, \emph{Regularity and asymptotic behaviour of the wave equation with a critical nonlinearity}, Ann. of Math. (2) 132 (1990), no. 3, 485--509.

\bibitem{GV95}
J.~Ginibre, G.~Velo, \emph{Generalized Strichartz inequalities for the wave equation}, J. Funct. Anal. 133 (1995), no. 1, 50--68.

\bibitem{G60}
E.~Gross, \emph{Quantum theory of interacting bosons}, Ann. Physics 9 (1960), 292--324.

\bibitem{GKO18}
M.~Gubinelli, H.~Koch, T.~Oh, \emph{Renormalization of the two-dimensional stochastic nonlinear wave equations}, Trans. Amer. Math. Soc. 370 (2018) no. 10, 7335--7359.

\bibitem{GKOT22}
M.~Gubinelli, H.~Koch, T.~Oh, L.~Tolomeo, \emph{Global dynamics for the two-dimensional stochastic nonlinear wave equations}, Int. Math. Res. Not. IMRN (2022), no. 21, 16954--16999.


\bibitem{GKO23}
M.~Gubinelli, H.~Koch, T.~Oh, \emph{Paracontrolled approach to the three-dimensional stochastic nonlinear wave equation with quadratic nonlinearity}, 
%
J. Eur. Math. Soc. (JEMS) 26 (2024), no. 3, 817--874.


\bibitem{H28}
D.~Hartree,
\emph{The Wave Mechanics of an Atom with a Non-Coulomb Central Field. Part I. Theory and Methods},
Theory and methods. Proc. Comb. Phil. Soc. 24, 89--132 (1928).

\bibitem{H00}
K.~Hidano,
\emph{Small data scattering and blow-up for a wave equation with a cubic convolution}, Funkcial. Ekvac. 43 (2000), no. 3, 559--588.

\bibitem{IM07}
S.~Ibrahim, M.~Majdoub, N.~Masmoudi, \emph{Ill-posedness of $H^1$-supercritical waves}, C. R. Math. Acad. Sci. Paris 345 (2007), no. 3, 133--138.

\bibitem{K94}
L.~Kapitanski, \emph{Global and unique weak solutions of nonlinear wave equations}, Math. Res. Lett. 1 (1994), no. 2, 211--223.

\bibitem{KT98}
M.~Keel, T.~Tao, \emph{Endpoint Strichartz estimates}, Amer. J. Math. 120 (1998), no. 5, 955--980.

\bibitem{KM08}
C.~Kenig, F.~Merle, \emph{Global well-posedness, scattering and blow-up for the energy-critical focusing non-linear wave equation}, Acta Math. 201 (2008), no. 2, 147--212.

\bibitem{KOPV12}
R.~Killip, T.~Oh, O.~Pocovnicu, M.~Vi\c{s}an, \emph{Global well-posedness of the Gross-Pitaevskii and cubic-quintic nonlinear Schr\"{o}dinger equations with non-vanishing boundary conditions}, Math. Res. Lett. 19 (2012), no. 5, 969--986.

\bibitem{KOC22}
J.~Kuan, T.~Oh, S.~\v{C}ani\'{c}, 
\emph{Probabilistic global well-posedness for a viscous nonlinear wave equation modeling fluid-structure interaction},
Appl. Anal. 101 (2022), no. 12, 4349--4373.



\bibitem{Latocca21}
M.~Latocca,
\emph{Almost sure existence of global solutions for supercritical semilinear wave equations},
J. Differential Equations 273 (2021), 83--121.



\bibitem{Lebeau05}
G.~Lebeau, \emph{Perte de r\'{e}gularit\'{e} pour les \'{e}quations d'ondes sur-critiques}, (French) [Loss of regularity for super-critical wave equations] Bull. Soc. Math. France 133 (2005), no. 1, 145--157.




\bibitem{LL01}
E.~Lieb and M.~Loss, \emph{Analysis},
Second edition
Grad. Stud. Math., 14
American Mathematical Society, Providence, RI, 2001. xxii+346 pp.



\bibitem{LS95}
H.~Lindblad, C.~Sogge, \emph{On existence and scattering with minimal regularity for semilinear wave equations}, J. Funct. Anal. 130 (1995), no. 2, 357--426

\bibitem{LM14}
J.~L\"{u}hrmann, D.~Mendelson.
\emph{Random data Cauchy theory for nonlinear wave equations of power-type on $\R^3$}, Comm. Partial Differential Equations 39 (2014), no. 12, 2262--2283. 

\bibitem{LM16}
J.~L\"{u}hrmann, D.~Mendelson. \emph{On the almost sure global well-posedness of energy sub-critical nonlinear wave equation $\R^3$}, New York J. Math. 22 (2016), 209--227.

\bibitem{NORS12}
A.~Nahmod, T.~Oh, L.~Rey-Bellet, G.~Staffilani, \emph{Invariant weighted Wiener measures and almost sure global well-posedness for the periodic derivative NLS}, J. Eur. Math. Soc. (JEMS) 14 (2012), no. 4, 1275--1330.



\bibitem{Mckean95}
H.P.~McKean, Statistical mechanics of nonlinear wave equations. IV. Cubic Schr\"odinger, Comm. Math.
Phys. 168 (1995), no. 3, 479--491.
 Erratum: Statistical mechanics of nonlinear wave equations. IV. Cubic
Schr\"odinger, Comm. Math. Phys. 173 (1995), no. 3, 675.





\bibitem{MS82}
G.~Menzala, W.~Strauss, 
\emph{On a wave equation with a cubic convolution}, J. Differential Equations 43 (1982), no. 1, 93--105.

\bibitem{M89}
K.~Mochizuki, 
\emph{On small data scattering with cubic convolution nonlinearity}, J. Math. Soc. Japan 41 (1989), no. 1, 143--160.



\bibitem{MXZ08}
C.~Miao, G,~Xu, L.~Zhao,
\emph{The Cauchy problem of the Hartree equation}, J. Partial Differential Equations 21(2008), no.1, 22--44.


\bibitem{MXZ11}
C.~Miao, G,~Xu, L.~Zhao, 
\emph{Global well-posedness and scattering for the energy-critical, defocusing Hartree equation in $\R^{1+n}$}, Comm. Partial Differential Equations 36 (2011), no. 5, 729--776.

\bibitem{MZZ14}
C.~Miao, J.~Zhang, J.~Zheng, \emph{The defocusing energy-critical wave equation with a cubic convolution}, Indiana Univ. Math. J. 63 (2014), no. 4, 993--1015.




\bibitem{MZZ15}
C.~Miao, J.~Zhang, J.~Zheng, \emph{Scattering theory for the radial $\dot H^{\frac12}$-critical wave equation with a cubic convolution}, J. Differential Equations 259 (2015), no. 12, 7199--7237.




\bibitem{Oh09}
T.~Oh, \emph{Invariant Gibbs measures and a.s. global well-posedness for coupled KdV systems}, Differential Integral Equations 22 (2009), no. 7-8, 637--668. 

\bibitem{OO20}
T.~Oh, M.~Okamoto, \emph{On the stochastic nonlinear Schr\"{o}dinger equations at critical regularities}, Stoch. Partial Differ. Equ. Anal. Comput. 8 (2020), no.4, 869--894.


\bibitem{OOP19}
T.~Oh, M.~Okamoto, O.~Pocovnicu, \emph{On the probabilistic well-posedness of the nonlinear Schr\"odinger equations with non-algebraic nonlinearities}, Discrete Contin. Dyn. Syst. 39 (2019), no. 6, 3479--3520.

\bibitem{OOPT23}
T.~Oh, M.~Okamoto, O.~Pocovnicu, N.~Tzvetkov, 
\emph{A remark on randomization of a general function of negative regularity},  
Proc. Amer. Math. Soc. Ser. B 11 (2024), 538--554.


\bibitem{OOR20}
T.~Oh, M.~Okamoto, T.~Robert, \emph{A remark on triviality for the two-dimensional stochastic nonlinear wave equation}, Stochastic Process. Appl. 130 (2020), no. 9, 5838--5864.

\bibitem{OOTz}
T.~Oh, M.~Okamoto, N.~Tzvetkov, \emph{Uniqueness and non-uniqueness of the Gaussian free field evolution under the two-dimensional Wick ordered cubic wave equation},  
Ann. Inst. Henri Poincar\'e Probab. Stat.
60 (2024), no. 3, 1684--1728.


\bibitem{OP16}
T.~Oh, O.~Pocovnicu, \emph{Probabilistic global well-posedness of the energy-critical defocusing quintic nonlinear wave equation on $\R^3$}, J. Math. Pures Appl. (9) 105 (2016), no. 3, 342--366.


\bibitem{OP17}
T.~Oh, O.~Pocovnicu,
\emph{A remark on almost sure global well-posedness of the energy-critical defocusing nonlinear wave equations in the periodic setting}, Tohoku Math. J. 69 (2017), no.3,  455--481.


\bibitem{OPT22}
T.~Oh, O.~Pocovnicu, N.~Tzvetkov, \emph{Probabilistic local well-posedness of the cubic nonlinear wave equation in negative Sobolev spaces}, Ann. Inst. Fourier (Grenoble) 72 (2022) no. 2, 771--830.


\bibitem{OPW20}
T.~Oh, O.~Pocovnicu, Y.~Wang, \emph{On the stochastic nonlinear Schr\"{o}dinger equations with nonsmooth additive noise}, 
Kyoto J. Math. 60 (2020), no. 4, 1227--1243.


\bibitem{OPW21}
T.~Oh, T.~Robert, P.~Sosoe, Y.~Wang, \emph{On the two-dimensional hyperbolic stochastic sine-Gordon equation}, Stoch. Partial Differ. Equ. Anal. Comput. 9 (2021), no. 1, 1--32.


\bibitem{ORT23}
T.~Oh, T.~Robert, N.~Tzvetkov, \emph{Stochastic nonlinear wave dynamics on compact surfaces}, Ann. H. Lebesgue 6 (2023), 161--223.

\bibitem{OWZ22}
T.~Oh, Y.~Wang, Y.~Zine, \emph{Three-dimensional stochastic cubic nonlinear wave equation with almost space-time white noise}, Stoch. Partial Differ. Equ. Anal. Comput. 10 (2022), no. 3, 898--963. 


\bibitem{P17}
O.~Pocovnicu, \emph{Almost sure global well-posedness for the energy-critical defocusing nonlinear wave equation on $\mathbb{R}^ d$, $d=4$ and $5$}, J. Eur. Math. Soc. (JEMS) 19 (2017), no. 8, 2521--2575.

\bibitem{RY99} 
D.~Revuz, M.~Yor,
{\it Continuous martingales and Brownian motion}, 
Third edition.
Grundlehren Math. Wiss., 293[Fundamental Principles of Mathematical Sciences]
Springer-Verlag, Berlin, 1999. xiv+602 pp.

\bibitem{SS94}
J.~Shatah, M.~Struwe, \emph{Well posedness in the energy space for semilinear wave equation with critical growth}, Int. Math. Res. Not. IMRN (1994), no. 7, 303--309.

\bibitem{S68}
W.~Strauss, \emph{Decay and asymptotics for $\Box u=F(u)$}, J. Functional Analysis 2 (1968), 409--457.

\bibitem{S88}
M.~Struwe, \emph{Globally regular solutions to the $u^5$ Klein-Gordon equation},
 Ann. Scuola Norm. Sup. Pisa Cl. Sci. (4) 15 (1988), no. 3, 495--513. 




 
\bibitem{SX16}
C.~Sun, B.~Xia, 
\emph{Probabilistic well-posedness for supercritical wave equations with periodic boundary condition on dimension three}, 
Illinois J. Math. 60(2016), no.2, 481--503.





\bibitem{Tao06}
T.~Tao, \emph{Spacetime bounds for the energy-critical nonlinear wave equation in three spatial dimensions}, Dyn. Partial Differ. Equ. 3 (2006), no. 2, 93--110.





\bibitem{TVZ07}
T.~Tao, M.~Vi\c{s}an, X.~Zhang,
\emph{The nonlinear Schr\"{o}dinger equation with combined power-type nonlinearities}, Comm. Partial Differential Equations 32 (2007), no. 7-9, 1281--1343.

\bibitem{TT10}
L.~Thomann, N.~Tzvetkov, \emph{Gibbs measure for the periodic derivative nonlinear
Schr\"{o}dinger equation}, Nonlinearity 23 (2010), no. 11, 2771--2791.

\bibitem{Tz06}
N.~Tzvetkov,
\emph{Invariant measures for the nonlinear Schr\"{o}dinger equation on the disc}, Dyn. Partial Differ. Equ. 3 (2006), no. 2, 111--160. 


\bibitem{V07}
M.~Vi\c{s}an, \emph{The defocusing energy-critical nonlinear Schr\"{o}dinger equation in higher dimensions}, Duke Math. J. 138 (2007), no. 2, 281--374.


\bibitem{Zhang23}
D.~Zhang, \emph{Stochastic nonlinear Schr\"{o}dinger equations in the defocusing mass and energy
critical cases},
Ann. Appl. Probab. 33 (2023), no. 5, 3652--3705.

\end{thebibliography}
\end{document}